\documentclass[11pt]{amsart}
\usepackage[margin=1in]{geometry}
\usepackage{bbm}
\usepackage{float, graphicx}
\usepackage[]{epsfig}
\usepackage{amsmath, amsthm, amssymb}
\usepackage{yhmath}
\usepackage{epsfig}
\usepackage{verbatim}
\usepackage{multicol}
\usepackage{url}
\usepackage{aliascnt}
\usepackage{latexsym}
\usepackage{mathrsfs}
\usepackage[colorlinks, bookmarks=true]{hyperref}
\usepackage{graphicx}
\usepackage{subcaption}
\usepackage{amsmath}
\usepackage{enumerate}
\usepackage[normalem]{ulem}
\usepackage{bm}
\usepackage{dsfont}
\usepackage{stmaryrd}
\usepackage{enumitem}
\usepackage{soul}
\usepackage[noadjust]{cite}
\usepackage{color}
\usepackage{xcolor}
\usepackage{hyperref}
\usepackage{indentfirst}
\usepackage[capitalise]{cleveref}

\numberwithin{equation}{section}
\usepackage{mathtools}
\newcommand{\trace}{\mathrm{tr}}

\newcommand{\cX}{{\mathcal X}}
\newcommand{\cY}{{\mathcal Y}}
\newcommand{\cT}{{\mathcal T}}
\newcommand{\cA}{{\mathcal A}}

\newcommand{\cD}{{\mathcal D}}
\newcommand{\cG}{{\mathcal G}}

\newcommand{\bN}{{\mathbb{N}}}
\newcommand{\bR}{{\mathbb R}}
\newcommand{\bS}{{\mathbb S}}

\newcommand{\bZ}{{\mathbb Z}}

\def\<{\langle}
\def\>{\rangle}

\def\cT{\mathcal{T}}
\def\cV{\mathcal{V}}

\def\cG{\mathcal{G}}

\def\cA{\mathcal{A}}
\def\cH{\mathcal{H}}
\def\cQ{\mathcal{Q}}

\def\Lamp[#1]{\boldsymbol{\Lambda}_{\mathrm{AMP}}^{(#1)}}
\def\lalg[#1]{\Lambda_{\mathrm{alg}, #1}}

\newcommand{\proj}{\mathbf{P}}
\newcommand{\bfS}{\mathbf{S}}

\def\cR{\mathcal{R}}
\def\cP{\mathcal{P}}
\def\cK{\mathcal{K}}

\def\supp{{\rm supp}}

\def\de{{\rm d}}

\def\BB{\mathbb{B}}

\def\TT{\mathbb{T}}

\def\calZ{\mathcal{Z}}

\newcommand{\RN}[1]{%
  \textup{\uppercase\expandafter{\romannumeral#1}}%
}

\newcommand{\RNum}[1]{\uppercase\expandafter{\romannumeral #1\relax}}

\newcommand{\Vol}{\operatorname{Vol}}
\newcommand{\Ent}{\operatorname{Ent}}
\newcommand{\dive}{\operatorname{div}}
\newcommand{\gradW}{\nabla\mkern-10mu\nabla}
\newcommand{\bW}{\operatorname{W}}

\theoremstyle{plain} 
\newtheorem{theorem}{Theorem}[section]
\newtheorem*{theorem*}{Theorem}

\newaliascnt{proposition}{theorem}
\newtheorem{proposition}[proposition]{Proposition}
\aliascntresetthe{proposition}
\crefname{proposition}{Proposition}{Propositions}
\Crefname{proposition}{Proposition}{Propositions}

\newaliascnt{lemma}{theorem}
\newtheorem{lemma}[lemma]{Lemma}
\aliascntresetthe{lemma}
\crefname{lemma}{Lemma}{Lemmas}
\Crefname{lemma}{Lemma}{Lemmas}

\newaliascnt{corollary}{theorem}
\newtheorem{corollary}[corollary]{Corollary}
\aliascntresetthe{corollary}
\crefname{corollary}{Corollary}{Corollaries}
\Crefname{corollary}{Corollary}{Corollaries}

\newaliascnt{definition}{theorem}

\aliascntresetthe{definition}
\crefname{definition}{Definition}{Definitions}
\Crefname{definition}{Definition}{Definitions}

\newaliascnt{remark}{theorem}

\aliascntresetthe{remark}
\crefname{remark}{Remark}{Remarks}
\Crefname{remark}{Remark}{Remarks}

\newaliascnt{example}{theorem}

\aliascntresetthe{example}
\crefname{example}{Example}{Examples}
\Crefname{example}{Example}{Examples}

\newaliascnt{problem}{theorem}

\aliascntresetthe{problem}
\crefname{problem}{Problem}{Problems}
\Crefname{problem}{Problem}{Problems}

\newaliascnt{assumption}{theorem}
\newtheorem{assumption}[assumption]{Assumption}
\aliascntresetthe{assumption}
\crefname{assumption}{Assumption}{Assumptions}
\Crefname{assumption}{Assumption}{Assumptions}

\crefname{claim}{Claim}{Claims}
\Crefname{claim}{Claim}{Claims}

\crefname{experiment}{Experiment}{Experiments}
\Crefname{experiment}{Experiment}{Experiments}

\title[Sharp Convergence of Wasserstein Gradient Flows for Spectrally Nonnegative Interaction Energies]{Sharp Convergence of Wasserstein Gradient Flows for Spectrally Nonnegative Interaction Energies}

\author[Z. Lin]{Zhengjiang Lin}
\address{(ZL) Department of Mathematics, Massachusetts Institute of Technology, 77 Massachusetts Ave, 02139 Cambridge MA, USA} 
\email{linzj@mit.edu}

\author[P. Rigollet]{Philippe Rigollet}
\address{(PR) Department of Mathematics, Massachusetts Institute of Technology, 77 Massachusetts Ave, 02139 Cambridge MA, USA} 
\email{rigollet@math.mit.edu}

\begin{document}

\begin{abstract}
    We study the long-time behavior of Wasserstein gradient flows for interaction energies
\[
\mathsf E[\mu]
=
\frac12\iint_{M\times M}K(x,y)\,\mathrm d\mu(x)\,\mathrm d\mu(y)
\]
on a closed manifold $M$. For kernels diagonal in a Laplace eigenbasis with nonnegative spectral coefficients, we prove a differential inequality relating the relative entropy to the energy gap. Consequently, for any nonnegative initial density $u_0\in L^p(M)$, $p>1$, the energy gap is integrable in time and satisfies
\[
\mathsf E[\mu_t]-\mathsf E_{\min}=o(t^{-1}).
\]
If all spectral coefficients are positive, the flow converges weakly to the constant measure. These
interaction energies need not be geodesically convex in Wasserstein space, and the associated flows
contain no diffusion; their global convergence therefore does not follow from standard Wasserstein
gradient flow theory. The kernels covered by our results include zonal kernels on spheres, kernels
arising in transformer models, regularized Riesz kernels, and inverse fractional Laplacian kernels.

We also investigate the sharpness of the $o(t^{-1})$ rate. For any smooth kernel in this class with infinitely many positive spectral coefficients and any $\delta>0$, we construct a solution of the linearized flow whose energy is comparable to $t^{-1-\delta}$ along a sequence of times tending to infinity. Moreover, for any $\delta>0$, by choosing a suitable inverse fractional Laplacian kernel
on the flat torus, we construct an exact solution of the nonlinear Wasserstein gradient flow whose
energy is comparable to $t^{-1-\delta}$. The nonlinear construction is based on uniform-in-time estimates
for the evolution of the dyadic Fourier coefficient blocks and a blockwise energy-persistence argument. These estimates also yield a uniform-in-time quantitative comparison between the nonlinear
Wasserstein gradient flow and its linearization.
\end{abstract}

\maketitle


\section{Introduction}

Energy optimization over probability measures is a classical problem in potential theory and discrete geometry. Let $M$ be a closed Riemannian manifold of dimension $d\geq1$, with metric $\langle\cdot,\cdot\rangle$ and volume measure $\de x$, and let $\cP(M)$ denote the set of Borel probability measures on $M$. Given a continuous symmetric kernel $K\colon M\times M\to\bR$, one seeks the measures that minimize or maximize the interaction energy
\begin{align}\label{e:energy}
    \mathsf{E}[\mu]
    \coloneqq
    \frac{1}{2}\iint_{M\times M}K(x,y)\,\de\mu(x)\,\de\mu(y),
    \qquad \mu\in\cP(M).
\end{align}
At the discrete level, taking $\mu$ to be an empirical measure of the form
$N^{-1}\sum_{i=1}^N\delta_{x_i}$ yields the corresponding $N$-point energy and turns the continuum variational problem into the search for well-distributed point configurations. This framework encompasses several classical problems, including the Thomson problem of minimizing the Coulomb energy of $N$ electrons on $\bS^2$ \cite{thomson1904structure}, Smale's seventh problem concerning the efficient construction of nearly optimal logarithmic-energy configurations on $\bS^2$ \cite{smale1998mathematical}, and the broader study of minimal Riesz-energy configurations and Fekete points \cite{hardin2005minimal}. Originating in electrostatics and potential theory, these problems are
now closely connected with approximation theory, discrepancy theory,
and the discretization of manifolds. Global optimizers have been
identified for several important families of such energies, including
distance and inner-product energies on spheres
\cite{bjorck1956distributions,tan2017energy,bilyk2019geodesic}. In the
setting considered here, nonnegativity of the spectral
coefficients in \Cref{a:kernel} implies that the uniform measure is a
global minimizer; see \eqref{e:energy expansion}.

Once a minimizer has been identified, a natural way to seek it
dynamically is to consider the steepest descent of $\mathsf{E}[\mu]$ on
the space of probability measures endowed with the Wasserstein metric. This leads to Wasserstein gradient flows, which provide a framework for
both nonlinear evolution equations and optimization over probability
measures; see
\cite{ambrosio2008gradient,jordan1998variational,otto2001geometry}
for further background. Their long-time behavior is well understood when
$\mathsf{E}[\mu]$ is convex along Wasserstein geodesics. In the absence of
geodesic convexity, however, even qualitative convergence to a minimizer
eludes standard gradient flow theory. This difficulty is
especially pronounced for the nondiffusive equation studied here: the
evolution is a pure transport flow, and the energy monotonicity identity
may fail to provide coercive control in regions where the density
vanishes. These are the central questions addressed in this paper.

To formulate the problem precisely, we first observe that the first variation of \eqref{e:energy} is given by
\begin{align}\label{e:Ku}
    \frac{\delta\mathsf{E}[\mu]}{\delta\mu}(x)
    =
    \cK[\mu](x)
    \coloneqq
    \int_M K(x,y)\,\de\mu(y).
\end{align}
For a function $u$ on $M$, we use the same symbols $\mathsf{E}[u]$ and $\cK[u]$, with $\de\mu$ replaced by $u\,\de x$. We also write $u\in\cP(M)$ if $u\,\de x\in\cP(M)$. Since $M$ is compact and $K$ is continuous, $\mathsf E$ attains its minimum on $\cP(M)$, which we denote by
\begin{align}\label{e:energy minimum}
    \mathsf{E}_{\min}
    \coloneqq
    \min_{\mu\in\cP(M)}\mathsf{E}[\mu].
\end{align}
The descending Wasserstein gradient flow of \eqref{e:energy} is formally given by the continuity equation
\begin{align}\label{e:wass grad flow}
    \partial_t\mu_t
    -
    \dive\bigl(\mu_t\gradW\mathsf{E}[\mu_t]\bigr)
    =
    0,
    \qquad
    \mu_0\in\cP(M),
\end{align}
where the driving vector field is the \emph{Wasserstein gradient}
\begin{align}\label{e:wass gradient}
    \gradW\mathsf{E}[\mu](x)
    \coloneqq
    \nabla\frac{\delta\mathsf{E}[\mu]}{\delta\mu}(x)
    =
    \nabla\cK[\mu](x).
\end{align}
Here $\nabla f$ denotes the Riemannian gradient of a function $f$, and $\dive\coloneqq\trace\nabla$ denotes the divergence operator. Equation~\eqref{e:wass grad flow} belongs to the class of nonlocal interaction equations arising in aggregation dynamics \cite{carrillo2011global}. 

Along every sufficiently regular solution, a direct computation gives
\begin{align}\label{e:mono energy}
    \frac{\de}{\de t}\mathsf{E}[\mu_t]
    =
    -
    \int_M
    \left|
        \gradW\mathsf{E}[\mu_t](x)
    \right|^2
    \,\de\mu_t(x)
    \leq 0,
\end{align}
so $\mathsf E$ is a natural Lyapunov functional. Nevertheless,
\eqref{e:mono energy} alone generally neither guarantees convergence
of $\mu_t$ to an energy minimizer nor provides a quantitative decay
rate for the gap $\mathsf E[\mu_t]-\mathsf E_{\min}$, even when the
minimizer is known explicitly. Besides the lack of geodesic convexity
of $\mathsf E$ in Wasserstein space, a central difficulty is that the
dissipation is weighted by the evolving measure $\mu_t$: it provides
no uniform, unweighted control of $\nabla\cK[\mu_t]$ where the density
is small or vanishes. When the density $u_t$ is bounded below by a
positive constant uniformly in space and time, this degeneracy
disappears, and established energy--dissipation arguments become
available. For kernels with suitable spectral structure, the resulting
control of $\|\nabla\cK[\mu_t]\|_{L^2(M)}$ can be combined with
interpolation and uniform regularity estimates to close a differential
inequality for the energy gap; see, for example,
\cite[Section~3.2]{chizat2026quantitative}.

We use this mechanism for the upper bound in the separate sharpness
construction of \Cref{thm:wass -1-delta rate intro}. There, uniform
perturbative estimates ensure that the constructed solution satisfies
$u_t\ge1/2$ for all $t\ge0$, and
\Cref{lem:negative Sobolev Besov interpolation} then yields
\eqref{e:diff ineql E ft intro}. The purpose of that construction,
however, is to exhibit a particular solution with matching upper and
lower decay bounds, not to establish convergence for general initial
data. Our global convergence theorem makes no positive lower-bound
or smallness assumption on the initial density.

Nor can such a lower bound be recovered here by simply allowing the
flow to evolve for a short time. For uniformly parabolic
drift--diffusion equations, diffusion instantaneously spreads mass
throughout a connected domain, and quantitative heat-kernel estimates
can provide positive lower bounds uniform in space and time after a
positive waiting time in compact settings with uniformly bounded
drift~\cite{Aronson1968,chizat2026convergence}. These bounds can, in turn, make
energy--dissipation arguments applicable even when the initial density
vanishes. By contrast, \eqref{e:wass grad flow} contains no diffusion:
vacuum regions are transported by the characteristic flow and persist
at every finite time. Establishing global convergence therefore
requires a mechanism that remains effective in the presence of this
degeneracy.

Our first main discovery is that entropy serves as a second Lyapunov
functional, even though it is absent from the energy \eqref{e:energy}
and the flow \eqref{e:wass grad flow} contains no diffusion. Its
dissipation arises instead from the spectral compatibility between the kernel
$K$ and the Laplace--Beltrami operator. More specifically, the entropy
satisfies a differential inequality that controls the full
interaction-energy gap, including the component not controlled by
\eqref{e:mono energy}. The resulting estimate is global, requires no
positive lower bound on the density, and yields
$\mathsf E[\mu_t]-\mathsf E_{\min}=o(t^{-1})$. Thus, whenever the minimizer is unique, the solution converges weakly
to the uniform measure even if the initial density vanishes on a region
of arbitrarily large volume. This convergence to a fully supported
limit occurs through transport alone: measure-zero regions are carried by
the flow and persist at every finite time.
\Cref{fig:exponential kernel wass flow} illustrates this relaxation
from an initial density that is not strictly positive.

\begin{figure}[t]
    \centering
    \begin{subfigure}[t]{0.2\textwidth}
        \centering
        \includegraphics[width=\linewidth]
        {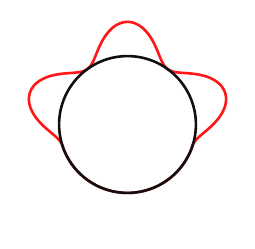}
        \caption{$t=0$}
    \end{subfigure}
    \hfill
    \begin{subfigure}[t]{0.2\textwidth}
        \centering
        \includegraphics[width=\linewidth]
        {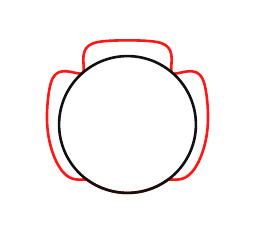}
        \caption{$t=0.05$}
    \end{subfigure}
    \hfill
    \begin{subfigure}[t]{0.2\textwidth}
        \centering
        \includegraphics[width=\linewidth]
        {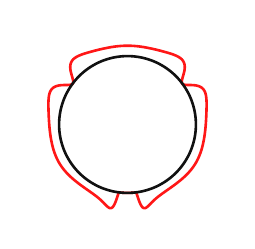}
        \caption{$t=0.2$}
    \end{subfigure}
    \hfill
    \begin{subfigure}[t]{0.2\textwidth}
        \centering
        \includegraphics[width=\linewidth]
        {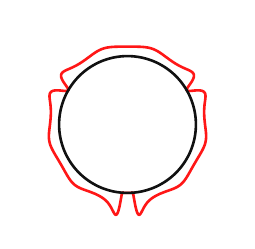}
        \caption{$t=20$}
    \end{subfigure}
    \caption{
        Evolution of the Wasserstein gradient flow
        \eqref{e:wass grad flow} on $M=\bS^1$ for the exponential kernel
        $K(x,y)=e^{\beta(x\cdot y)}$ with $\beta=4$, as introduced in
        \Cref{s:sphere kernel}. The red curves show the density $u_t$ on the unit circle (black). The displayed times are $t=0,\,0.05,\,0.2,\,20$, and the corresponding energy gaps $\mathsf E[u_t]-\mathsf E_{\min}$ are $3.3$, $1.3$, $6.4\times10^{-2}$, $1.8\times10^{-7}$. Relative to the normalized volume measure on $\bS^1$, the initial density is $u_0(\theta) = \frac{5}{3}\bigl(1+\sin(5\theta)\bigr) \mathbf{1}_{[-\pi/10, \,11\pi/10]}(\theta)$. In particular, $u_0$ vanishes on the open arc $(11\pi/10,19\pi/10)$. Nevertheless, $u_t$ approaches the uniform equilibrium $\sigma = 1$. 
    }
    \label{fig:exponential kernel wass flow}
\end{figure}

Our second main contribution concerns the quantitative long-time behavior
of the nondiffusive flow \eqref{e:wass grad flow}. For inverse fractional
Laplacian kernels on the flat torus, we establish uniform-in-time Besov-norms estimates for the nonlinear perturbation and quantitative lower bounds
for its dyadic Fourier coefficient blocks up to their corresponding linear decay time
scales. By distributing the initial energy across these blocks and
combining their persistence with an interpolation inequality, we
construct, for every $\delta>0$, an exact nonlinear solution of \eqref{e:wass grad flow} whose energy
has the two-sided decay rate $t^{-1-\delta}$. The same estimates also
yield a uniform-in-time quantitative comparison between the nonlinear
flow \eqref{e:wass grad flow} and its linearization.

Our structural assumption is that the kernel is simultaneously
diagonalizable with the Laplacian operator and has nonnegative
coefficients on all nonconstant modes.

\begin{assumption}[Spectral compatibility and nonnegativity]\label{a:kernel}
The kernel $K$ satisfies the following conditions.
\begin{enumerate}
    \item \textbf{Regularity and symmetry.}
    The kernel $K$ belongs to $C^2(M\times M)$ and is symmetric:
    $K(x,y)=K(y,x)$ for all $x,y\in M$.

    \item \textbf{Spectral compatibility.}
    There exists a real orthonormal basis $\{\phi_n\}_{n\ge0}$ of
    $L^2(M)$ consisting of eigenfunctions of $-\Delta$,
    \[
        -\Delta\phi_n=\lambda_n\phi_n,
        \qquad
        0=\lambda_0<\lambda_1\le\lambda_2\le\cdots,
    \]
    in which $K$ admits the diagonal expansion
    \[
        K(x,y)=\sum_{n=0}^{\infty}a_n\phi_n(x)\phi_n(y).
    \]

    \item \textbf{Spectral nonnegativity.}
    The coefficients on all nonconstant modes are nonnegative:
    \[
        a_n\ge0 \qquad \text{for every } n\ge1.
    \]
\end{enumerate}
\end{assumption}


The constant-mode coefficient $a_0$ clearly plays no role in the dynamics \eqref{e:wass grad flow}. Moreover, \Cref{a:kernel} alone does not imply the convexity $\frac{\de^2}{\de t^2}\mathsf{E}[\mu_t]\geq 0$. The assumption encompasses smooth zonal kernels on spheres, regularizations of singular Riesz kernels, and smooth kernels associated with inverse fractional powers of the Laplace--Beltrami operator. These examples and their applications are discussed in \Cref{s:kernel examples}.

Throughout the paper, we work under \Cref{a:kernel}, and we also assume that 
    \begin{align}\label{a:normalized uniform measure}
        1=\Vol(M)=\int_M1\,\de x,
    \end{align}
so that $\sigma\coloneqq1\de x\in\cP(M)$. We call $\sigma$ the uniform measure on $M$. Without this normalization assumption, all the results of this paper remain valid after a constant rescaling of the metric on $M$. We denote by $\|\cdot\|_p$ the $L^p$ norm with respect to the volume measure. Under \eqref{a:normalized uniform measure}, we have $\phi_0=1$. For $\mu\in\cP(M)$, we denote $b_n\coloneqq\int_M\phi_n\,\de\mu$. Then $b_0=1$, and
\begin{align}\label{e:energy expansion}
    \mathsf{E}[\mu]
    &=
    \frac{1}{2}\sum_{n=0}^{\infty}a_n b_n^2
    \geq
    \frac{a_0b_0^2}{2}
    =
    \frac{a_0}{2}
    =
    \mathsf{E}_{\min}.
\end{align}
In particular, $\sigma$ is an energy minimizer. Equivalently,
\begin{align*}
    \mathsf E[\mu]-\mathsf E_{\min}
    =
    \frac12\iint_{M\times M}
    K(x,y)\,\de(\mu-\sigma)(x)\,\de(\mu-\sigma)(y).
\end{align*}
This energy gap is the modulated interaction energy of $\mu$ relative
to $\sigma$ \cite{serfaty2026lectures}, and is also known as the
maximum mean discrepancy (MMD) between $\mu$ and $\sigma$ in the
kernel-methods literature
\cite{gretton2012kernel,li2015generative,li2017mmd}. In the
uniform-target setting considered here, \eqref{e:wass grad flow} is
therefore also the Wasserstein gradient flow of this modulated energy. The study of nonlocal interaction energies and their Wasserstein gradient-flow structure has a long history, including
substantial work on singular Coulomb and Riesz interactions; see, for
example,
\cite{mccann1997convexity,carrillo2003kinetic,carrillo2006contractions,
bertozzi2007finite,bertozzi2009blowup,carrillo2011global,
serfaty2020meanfield,lin2024nonlocal,boufadene2025global,
de2026wasserstein,chizat2026quantitative,serfaty2026lectures,rosenzweig2026wasserstein}.
Wasserstein gradient flows of MMD functionals have also recently
attracted considerable attention in machine learning as particle-based
methods for distribution matching, neural generative modeling, and
posterior sampling
\cite{arbel2019maximum,altekruger2023neural,hertrich2024generative,
hagemann2024posterior,galashov2025deep}.

\subsection{Main results}

The principal contributions of this paper are the following.

\begin{enumerate}
    \item In \Cref{thm: intro global convergence,thm: intro mono entropy}, we establish an entropy inequality for the pure Wasserstein
    transport flow \eqref{e:wass grad flow}. For every nonnegative initial
    density $u_0\in L^p(M)$ with $p>1$, this inequality yields global
    energy convergence at the universal rate $o(t^{-1})$, without
    requiring strict positivity, smallness, or high regularity. If the
    uniform measure is the unique minimizer, it further yields weak
    convergence of $\mu_t$ to the uniform measure. The inequality applies
    to every kernel satisfying \Cref{a:kernel}, including all the examples
    presented in \Cref{s:kernel examples}, and requires neither a maximum
    principle nor a kernel-specific formula.

    \item In \Cref{thm:linear-sharp-delta subseq intro,thm:linear-sharp-delta intro}, we establish the sharpness of the $o(t^{-1})$-rate at the level
    of the linearized equation. For any kernel with infinitely many
    positive modes in \Cref{a:kernel}, we construct a linearized solution whose energy decay
    is arbitrarily close to $t^{-1}$ along a sequence of times. We also
    construct a smooth kernel for which, for every $\delta>0$, a
    linearized solution satisfies the two-sided estimate 
    $\mathsf E[\overline f_t]\asymp t^{-1-\delta}$.
    Finally, we identify an obstruction to such global-in-time
    two-sided polynomial asymptotics for general kernels.

    \item In \Cref{thm:wass -1-delta rate intro}, for inverse fractional Laplacian kernels on the torus, we
     produce an exact nonlinear solution of \eqref{e:wass grad flow} satisfying the two-sided energy law
    $\mathsf{E}[u_t] -\mathsf E_{\min}\asymp t^{-1-\delta}$ through a suitable
    choice of the kernel and initial datum.
\end{enumerate}

The first contribution is developed in
\Cref{s:main results convergence rate}; the sharpness results and the
uniform-in-time nonlinear perturbation theory are presented in
\Cref{s:sharp convergence result}.

\subsubsection{Global convergence and entropy}\label{s:main results convergence rate}

\begin{theorem}\label{thm: intro global convergence}
    For any $p> 1$ and any $u_0 \in L^p(M) \cap \cP(M)$, let $\mu_t = u_t(x) \de x$ solve \eqref{e:wass grad flow} starting from $\mu_0 = u_0(x) \de x$. Then,
        \begin{align}\label{e:energy integrability}
            \int_0 ^{\infty} \left(\mathsf{E}[\mu_t] - \mathsf{E}_{\min} \right) \, \de t < \infty.
        \end{align}
    In particular, as $t \to \infty$,
        \begin{align}\label{e:energy convergence}
         \mathsf{E}[\mu_t] - \mathsf{E}_{\min} = o(t^{-1}).
        \end{align}
    Furthermore, every weak limit of $\mu_t$ must be a minimizer of $\mathsf{E}[\cdot]$, and if $\sigma \in \cP(M)$ is the unique minimizer of $\mathsf{E}[\cdot]$, then
        \begin{align}\label{e:global weak convergence}
            \lim_{t \to \infty} \bW_2(\mu_t , \sigma) = 0,
        \end{align}
    where $\bW_2(\cdot,\cdot)$ denotes the Wasserstein-2 distance.
\end{theorem}

The conclusion is genuinely global with respect to the initial datum: the
initial density need only be nonnegative and belong to $L^p(M)$ for some
$p>1$. In particular, no uniform positive lower bound, high Sobolev
regularity, or proximity to equilibrium is required. The initial density
may vanish on an open set of arbitrarily large volume. Since
\eqref{e:wass grad flow} is transported by a characteristic flow, such a
vacuum region is not instantaneously filled unlike in the diffusion case. Instead, it is carried by the flow at
every finite time. The convergence \eqref{e:global weak convergence} therefore comes from neither parabolic smoothing
nor a hidden pointwise lower bound.

The proof of \Cref{thm: intro global convergence} rests on an additional
Lyapunov functional that provides stronger control than the energy alone.
For $u\in L^p(M)\cap\cP(M)$, define its entropy by
\begin{align*}
    \Ent(u)
    \coloneqq
    \int_M u\log u\,\de x.
\end{align*}
The key ingredient is the following entropy inequality.

\begin{theorem}\label{thm: intro mono entropy}
    Adopt assumptions of \Cref{thm: intro global convergence}, we have
        \begin{align}\label{e:mono entropy}
            \frac{\de}{\de t} \Ent(u_t) \leq - 2\lambda_1 \left( \mathsf{E}[\mu_t] - \mathsf{E}_{\min} \right),
        \end{align}
    where $\lambda_1$ is the first nonzero Laplacian eigenvalue of $M$.
\end{theorem}

Formally, this inequality is easy to see: the
spectral decomposition in \Cref{a:kernel} gives the identity
\begin{align}\label{e:entropy spectral identity intro}
    \frac{\de}{\de t}\Ent(u_t)
    &=-\sum_{n=1}^{\infty}\lambda_n a_n b_n(t)^2
    \leq
    -\lambda_1\sum_{n=1}^{\infty}a_n b_n(t)^2
    =-2\lambda_1
    \left(\mathsf E[\mu_t]-\mathsf E_{\min}\right),
\end{align}
where $b_n(t)=\int_M\phi_n u_t\,\de x$. 
Thus, the nonnegativity of the spectral coefficients gives a direct,
unweighted coercive estimate for the full energy gap. No lower bound for
$u_t$ appears. Since $\Ent(u_t)\geq0$ by Jensen's inequality, and the energy gap is nonnegative and
nonincreasing, integrating \eqref{e:entropy spectral identity intro} first
gives \eqref{e:energy integrability}, and integrability together with
monotonicity yields
$\mathsf E[\mu_t]-\mathsf E_{\min}=o(t^{-1})$. This short argument is the
global mechanism behind \Cref{thm: intro global convergence}; in
particular, it simultaneously covers all kernels satisfying
\Cref{a:kernel}.

Several recent works obtain convergence through mechanisms different from ours. For entropy-regularized linearly convex energies of
the form $\cG(u)+\tau\Ent(u)$ with $\tau>0$,
\cite{chizat2026convergence} exploits the diffusion generated by the
entropy term to establish an $O(t^{-1})$ energy-decay rate, with faster
rates under stronger convexity assumptions. By contrast, \eqref{e:wass grad flow} contains no diffusion: entropy is not part of the energy \eqref{e:energy}, but instead serves as an auxiliary Lyapunov functional.

For Wasserstein gradient flows of MMD without diffusion, existing convergence
results either require additional pathwise, positivity, regularity, or
smallness assumptions, or exploit the exceptional structure of the
Coulomb kernel
\cite{arbel2019maximum,boufadene2025global,chizat2026quantitative,de2026wasserstein,rosenzweig2026wasserstein}.
The Coulomb kernel $\cK=(-\Delta)^{-1} \circ \Pi_0$, corresponding to $s=1$ in
\Cref{s:inverse laplacian kernel}, is distinguished by a transport maximum
principle. On the torus, \cite{chizat2026quantitative} uses this structure
to establish exponential energy decay for bounded initial densities and
bounded, uniformly positive targets. For arbitrary Borel probability
measures as initial data, \cite{de2026wasserstein} proves instantaneous
$L^\infty$-regularization and derives exponential decay toward bounded,
uniformly positive targets through a defective
Polyak--\L ojasiewicz inequality that explicitly accounts for vacuum
regions. For the inverse fractional Laplacian kernels in
\Cref{s:inverse laplacian kernel} with $s>1$,
\cite{chizat2026quantitative} establishes local polynomial convergence
when the initial density and the target are strictly positive and
sufficiently regular, and their initial discrepancy is sufficiently small.

\subsubsection{Sharpness and uniform-in-time nonlinear perturbation theory}
\label{s:sharp convergence result}

The entropy argument yields the universal $o(t^{-1})$-rate in
\Cref{thm: intro global convergence}, but \Cref{a:kernel} alone does not
imply any uniform improvement of this rate. This obstruction is already
visible in the linearization around the uniform measure
$\sigma=\de x$. Suppose that a solution $u_t^\varepsilon$ of
\eqref{e:wass grad flow} admits the expansion
\begin{align*}
    u_t^\varepsilon
    =
    1+\varepsilon \overline{f}_t+o(\varepsilon)
\end{align*}
as $\varepsilon\to0$. As shown in \Cref{s:linearization flow},
\begin{align*}
    \mathsf E[u_t^\varepsilon]-\mathsf E_{\min}
    =
    \varepsilon^2\mathsf E[\overline{f}_t]+o(\varepsilon^2),
\end{align*}
and the first-order perturbation $\overline{f}_t$ satisfies the
linearization of \eqref{e:wass grad flow}:
\begin{align}\label{e:linearized flow intro}
    \partial_t \overline{f}_t-\Delta\cK[\overline{f}_t]=0, \quad 
    \int_M \overline{f}_t\,\de x=0, \quad 
    t \in[0,\infty).
\end{align}

As follows from \eqref{e:linearized energy basis expansion}, if only finitely many
coefficients $a_n$ in the expansion of $K$ from \Cref{a:kernel} are
nonzero, then $\mathsf E[\overline{f}_t]$ decays exponentially as
$t\to\infty$. By contrast, if $K$ has infinitely many
positive coefficients, we show that the decay of $\mathsf E[\overline{f}_t]$ can be arbitrarily close to $t^{-1}$ along a sequence of times.

\begin{theorem}[Sharp $o(t^{-1})$-rate for the linearized flow for arbitrary kernels]
\label{thm:linear-sharp-delta subseq intro}
For any $\delta>0$ and any kernel $K$ satisfying \Cref{a:kernel} with
infinitely many positive $a_n$, there exists a solution $\overline{f}_t$ of the linearized equation
\eqref{e:linearized flow intro} in $H^{-1}(M)$ and a sequence of times
$\{t_k\}_{k\geq1}$ tending to infinity such that
\begin{align*}
    e^{-3}t_k^{-1-\delta}
    \leq
    \mathsf E[\overline{f}_{t_k}]
    \leq
    2t_k^{-1-\delta},
    \qquad k\geq1.
\end{align*}
\end{theorem}

The restriction to a sequence of times is unavoidable for general
kernels. Indeed, in \Cref{thm:lacunary no uniform rate}, we construct a
smooth kernel $K$ such that, for any initial datum
$f_0\in H^{-1}(M)$ and any $\beta>0$, the corresponding linearized
solution $\overline f_t$ satisfies
\begin{align*}
    \limsup_{t\to\infty} \, t^\beta \cdot \mathsf E[\overline{f}_t]<\infty
    \quad\Longrightarrow\quad
    \liminf_{t\to\infty} \, t^{5\beta/4} \cdot \mathsf E[\overline{f}_t]=0.
\end{align*}
Thus no solution $\overline{f}_t$ for this kernel can satisfy a two-sided polynomial
estimate $\mathsf E[\overline{f}_t]\asymp t^{-\beta}$ for all sufficiently large times.

On the other hand, we can construct a suitable smooth kernel which realizes a two-sided $o(t^{-1})$-rate uniformly for all
$t\geq0$.

\begin{theorem}[Sharp $o(t^{-1})$-rate for the linearized flow]
\label{thm:linear-sharp-delta intro}
There exists a kernel
$K\in C^\infty(M\times M)$ satisfying \Cref{a:kernel} with the following property: for any $\delta>0$, there is a smooth solution
$\overline{f}_t$ of the linearized equation \eqref{e:linearized flow intro} such that
\begin{align*}
    C^{-1}(t+1)^{-1-\delta}
    \leq
    \mathsf E[\overline{f}_t]
    \leq
    C(t+1)^{-1-\delta}, \quad \forall t \geq 0.
\end{align*}
Here, $C>1$ depends only on $\delta$.
\end{theorem}

We next seek to transfer this type of two-sided decay estimate to an
exact solution $u_t^\varepsilon$ of the nonlinear Wasserstein gradient flow \eqref{e:wass grad flow}. However, passing from the linearized construction to the nonlinear equation is not
a formal perturbation. A conventional stability estimate between $u_t^\varepsilon$ and $1+\varepsilon \overline{f}_t$ on a fixed
interval $[0,T]$ allows the error to grow with $T$. Such an estimate is
insufficient here because the linear signal that we seek to detect tends
to zero, and a coarse uniform error bound can therefore eventually dominate the desired asymptotic
profile.

We overcome this difficulty for the inverse fractional Laplacian kernels
$\cK=(-\Delta)^{-s}\circ\Pi_0$ introduced in
\Cref{s:inverse laplacian kernel} on the flat torus
$M=\TT^d=\left(\bR/(2\pi\bZ)\right)^{\otimes d}$ by analyzing the nonlinear perturbation equation \eqref{e:wass grad flow perturbation intro}, which is a perturbation of \eqref{e:wass grad flow} around the uniform measure. By these estimates, we obtain an exact nonlinear
solution with the two-sided decay stated in the following \Cref{thm:wass -1-delta rate intro} and
a uniform-in-time Besov-norm bound in \Cref{cor:uniform perturbation Besov norms}.

\begin{theorem}[$o(t^{-1})$-rate for the Wasserstein gradient flow \eqref{e:wass grad flow}]\label{thm:wass -1-delta rate intro}
Let $M=\TT^d$.
For any $s>1$ and $\sigma>d/2$, let
$\cK=(-\Delta)^{-s}\circ\Pi_0$.
There exists a solution $u_t$ of \eqref{e:wass grad flow} such that for any $t \geq 0$, $u_t\in B^\sigma(\TT^d)\cap\cP(\TT^d)$, and
\begin{align}\label{e:wass -1-delta rate}
C^{-1}(t+1)^{-\beta}
\leq
\mathsf{E}[u_t]-\mathsf{E}_{\min}
\leq
C(t+1)^{-\beta}.
\end{align}
Here $B^\sigma(\TT^d)$ is the Besov space defined in
\eqref{e:Besov norm def}, $\beta=\frac{s+\sigma}{s-1}$, and $C>1$ depends only on $s,\sigma,d$.
\end{theorem}

Given $\delta>0$, we can choose $s$ and $\sigma$ in \Cref{thm:wass -1-delta rate intro} so that
\begin{align*}
s>1+\frac{d/2+1}{\delta},
\qquad
\sigma=\delta(s-1)-1.
\end{align*}
Then $\sigma>d/2$ and $\beta=\frac{s+\sigma}{s-1}=1+\delta$. Thus, the rate in \eqref{e:wass -1-delta rate} becomes
$(1+t)^{-1-\delta}$, as in
\Cref{thm:linear-sharp-delta subseq intro,thm:linear-sharp-delta intro}.
Moreover, since $\sigma>d/2$, the embedding
$B^\sigma(\TT^d)\hookrightarrow C(\TT^d)$ in
\eqref{e:Besov norm control infinity norm} implies that the solution
constructed in \Cref{thm:wass -1-delta rate intro} has a continuous
density for every $t\geq0$.

We next describe the mechanism underlying the proof of
\Cref{thm:wass -1-delta rate intro}. We seek a solution of
\eqref{e:wass grad flow} of the form
$u_t=1+\varepsilon f_t$ for a mean-zero $f_t$. The perturbation
$f_t$ satisfies
\begin{align}\label{e:wass grad flow perturbation intro}
    \partial_t f_t-\Delta\cK[f_t]
    =
    \varepsilon
    \left\langle\nabla\cK[f_t],\nabla f_t\right\rangle
    +
    \varepsilon f_t\Delta\cK[f_t],
\end{align}
whereas the corresponding linearized solution $\overline f_t$ satisfies
\eqref{e:linearized flow intro}. By \eqref{e:energy ut ft},
\begin{align*}
    \mathsf E[u_t]-\mathsf E_{\min}
    =
    \varepsilon^2\mathsf E[f_t].
\end{align*}

The upper bound in \Cref{thm:wass -1-delta rate intro} does not depend
on the particular initial datum used in its proof.
Indeed, \Cref{prop:eulerian uniform estimate} applies to every
$f_0\in L_0^2(\TT^d)\cap B^\sigma(\TT^d)$ satisfying
\eqref{e:eulerian smallness condition} and gives
\begin{align*}
    \sup_{t\geq0}\|f_t\|_{B^\sigma}
    \lesssim  \|f_0\|_{B^\sigma}.
\end{align*}
The constant in \eqref{e:eulerian smallness condition} may be chosen small so that this estimate, together with the embedding
$B^\sigma(\TT^d)\hookrightarrow L^\infty(\TT^d)$, ensures
$u_t=1+\varepsilon f_t\geq1/2$ for all $t\geq0$. Combining \eqref{e:mono energy} with the interpolation inequality in
\Cref{lem:negative Sobolev Besov interpolation} yields
\begin{align}\label{e:diff ineql E ft intro}
    \frac{\de}{\de t}\mathsf E[f_t]
    \lesssim
    -
    \|f_0\|_{B^\sigma}^{-\frac{2(s-1)}{\sigma+s}}
    \mathsf E[f_t]^{\frac{\sigma+2s-1}{\sigma+s}}.
\end{align}
Integrating this differential inequality gives
\begin{align*}
    \mathsf E[f_t]
    \lesssim
    \|f_0\|_{B^\sigma}^2 \cdot 
    (1+t)^{-\frac{s+\sigma}{s-1}}.
\end{align*}
Equivalently,
\begin{align*}
    \mathsf E[u_t]-\mathsf E_{\min}
    \lesssim \varepsilon^2 \cdot \|f_0\|_{B^\sigma}^2 \cdot 
    (1+t)^{-\frac{s+\sigma}{s-1}}.
\end{align*}
Thus, every initial datum satisfying
\eqref{e:eulerian smallness condition} obeys the upper decay estimate in \Cref{thm:wass -1-delta rate intro}.

To construct a solution attaining the lower bound in \Cref{thm:wass -1-delta rate intro}, we first look at the
dyadic structure of the linearized equation \eqref{e:linearized flow intro}. For $k \in \bZ_+$, on Laplacian
eigenfunctions with eigenvalues comparable to $2^{2k}$, the linearized
generator $-\Delta\cK=(-\Delta)^{1-s}\circ\Pi_0$ has size comparable to $2^{-2(s-1)k}$. Hence the $k$-th dyadic Laplacian eigenfunction block
of $\overline f_t$ decays on the time scale $2^{2(s-1)k}$. We construct $f_0$ by choosing one
normalized Laplacian eigenfunction from each dyadic block and
assigning it the initial amplitude $2^{-\sigma k}$.  The corresponding linearized energy $ \mathsf{E}[\overline{f}_t]$ in \eqref{e:linearized energy basis expansion} is then
\begin{align*}
    \mathsf E[\overline f_t]
    \asymp
    \sum_{k=1}^{\infty}
    2^{-2(s+\sigma)k}
    \exp\left(-2\cdot2^{-2(s-1)k}t\right).
\end{align*}
The contribution of the $k$-th block remains comparable to
$2^{-2(s+\sigma)k}$ until times of order $2^{2(s-1)k}$. For any sufficiently large $t$, selecting the persistent $k$-th block with $k \sim \frac{\log_2 t}{2(s-1)}$ gives
\begin{align*}
\mathsf{E}[\overline{f}_t] \gtrsim 2^{-2(s+\sigma)k}
\asymp
t^{-(s+\sigma)/(s-1)}.
\end{align*}
This gives the desired lower bound at the level of the linearized equation.

The remaining difficulty is to show that this blockwise persistence
survives the nonlinear transport in
\eqref{e:wass grad flow perturbation intro}. Applying the smooth
Littlewood--Paley projector $\proj_k$ from
\eqref{e:partition of unity projection} yields the blockwise evolution
equation \eqref{e:projected Eulerian equation} for $\proj_k f_t$. Its
transport term is $\varepsilon \left\langle\nabla\cK[f_t],\nabla\proj_k f_t\right\rangle$, corresponding to the velocity
$v_t=\varepsilon\nabla\cK[f_t]$. The estimates
\eqref{e:dyadic persistence quantitative} and
\eqref{e:dyadic persistence simple} in
\Cref{prop:eulerian uniform estimate} show that, when $\varepsilon$ is
sufficiently small, the accumulated nonlinear errors do not destroy
the selected blocks before their linear decay time scales.
Consequently, the $L^2$-mass of the $k$-th block persists up to times
of order $2^{2(s-1)k}$, yielding the matching lower bound in
\Cref{thm:wass -1-delta rate intro}.

The uniform-in-time Besov norms bounds obtained in
\Cref{prop:eulerian uniform estimate} also yield a global
nonlinear-to-linear comparison. Indeed, as shown in
\eqref{e:ft-bar ft differential equation},
$f_t-\overline f_t$ solves the inhomogeneous linearized equation with
source $\varepsilon\dive(f_t\nabla\cK[f_t])$. Applying the linear
semigroup estimate in \Cref{lem:semigroup estimate} gives the following
result. The loss of one derivative is caused by $\nabla f_t$ in the nonlinear source term in \eqref{e:wass grad flow perturbation intro}.

\begin{corollary}[Uniform perturbation of the linearized solution]
\label{cor:uniform perturbation Besov norms}
Adopt the assumptions and notation of
\Cref{prop:eulerian uniform estimate}.
Let
$\overline{f}_t=e^{-t\cA}f_0$ be the solution of
\eqref{e:linearized flow in nonlinear section}.
Then there exists a constant $C$, depending only on $s,\sigma,d$, such that
\begin{align}\label{e:uniform perturbation Besov norms}
\sup_{T\geq 0}
\|f_t-\overline f_t\|_{\BB_T^{\sigma-1}}
\leq
C\varepsilon\|f_0\|_{B^\sigma}^2.
\end{align}
Here the modified Besov norm $\|\cdot\|_{\BB_T^{\sigma-1}}$ is defined in \eqref{e:uniform in time Besov norm}.
\end{corollary}

In the proofs of \Cref{thm:wass -1-delta rate intro} and
\Cref{prop:eulerian uniform estimate}, the only torus-specific ingredient
is the Fourier-series proof of the dyadic transport estimate in
\Cref{lem:torus dyadic transport}. The spectral construction of the
initial datum, the blockwise energy-persistence estimate, the bootstrap
argument, and the energy interpolation use only the Laplacian spectral
decomposition introduced in \Cref{s:preliminary besov norms}.
Consequently, the same construction extends to a closed manifold once an
analogue of \Cref{lem:torus dyadic transport}, together with the standard
well-posedness and propagation-of-regularity results used in
\cite{chizat2026quantitative}, is available on that manifold.

We conclude the introduction by discussing several examples of kernels
satisfying \Cref{a:kernel}.

\subsection{Examples of kernels satisfying \Cref{a:kernel}}\label{s:kernel examples}

\subsubsection{Zonal kernels on the sphere}\label{s:sphere kernel}

Let $M=\bS^d\subseteq\bR^{d+1}$ be the unit sphere equipped with the metric induced from $\bR^{d+1}$, and consider a zonal kernel
\begin{align*}
    K(x,y)=f\bigl(\langle x,y\rangle\bigr),
    \qquad x,y\in\bS^d,
\end{align*}
where $f\in C^2([-1,1])$ and $\langle x,y\rangle$ denotes the Euclidean inner product. For each $\ell\geq0$, let $\cH_\ell$ be the space of spherical harmonics of degree $\ell$. This is an eigenspace of $-\Delta_{\bS^d}$ with eigenvalue $\lambda_\ell=\ell(\ell+d-1)$. If $\{Y_{\ell,m}\}_{m=1}^{N(d,\ell)}$ is a real orthonormal basis of $\cH_\ell$, where $N(d,\ell)=\dim\cH_\ell$, then the Funk--Hecke formula implies that $\cK$ acts by a scalar $a_\ell$ on $\cH_\ell$. The corresponding spectral expansion of $K$ is
\begin{align}\label{e:sphere kernel expansion}
    K(x,y)
    =
    \sum_{\ell=0}^{\infty}
    a_\ell
    \sum_{m=1}^{N(d,\ell)}
    Y_{\ell,m}(x)Y_{\ell,m}(y),
\end{align}
as in part~(2) of \Cref{a:kernel}. According to Schoenberg's characterization of positive-definite zonal kernels on spheres \cite{schoenberg1942positive}, the condition $a_\ell\geq0$ for every $\ell\geq1$ can be characterized by the nonnegativity of the Gegenbauer expansion coefficients of $f$ for $d\geq2$ or the Chebyshev expansion coefficients of $f$ for $d = 1$. In particular, a convenient sufficient condition for \Cref{a:kernel}, valid in every dimension, is
\begin{align}\label{e:positive power series zonal kernel}
    f(t)
    =
    c+\sum_{n=1}^{\infty}c_nt^n,
    \quad c \in \bR, \,
    c_n\geq0,
    \quad
    \sum_{n=1}^{\infty}n^2c_n<\infty,
\end{align}
where the last condition ensures that $f\in C^2([-1,1])$.

A central example for the applications considered here is the exponential dot-product kernel
\begin{align}\label{e:transformer kernel}
    f(t)=e^{\beta t},
    \qquad \beta>0.
\end{align}
Its Gegenbauer expansion gives
\begin{align}\label{e:transformer kernel coefficients}
    a_\ell
    =
    \Vol(\bS^d)
    \Gamma\Bigl(\frac{d+1}{2}\Bigr)
    \Bigl(\frac{2}{\beta}\Bigr)^{\frac{d-1}{2}}
    I_{\ell+\frac{d-1}{2}}(\beta),
    \qquad \ell\geq0,
\end{align}
where $I_\nu$ denotes the modified Bessel function of the first kind. Since
\begin{align*}
    I_\nu(\beta)>0
    \quad
    \text{ for }\beta>0\text{ and }\nu\geq0,
\end{align*}
we have $a_\ell>0$ for every $\ell\geq0$. See
\cite[Proposition~3.4]{geshkovski2025mathematical} for further details. This exponential kernel \eqref{e:transformer kernel} arises in idealized continuous-depth models of self-attention. In particular, the unnormalized self-attention dynamics studied in \cite{geshkovski2025mathematical} are closely related to the Wasserstein gradient ascent associated with $K(x,y)=e^{\beta\langle x,y\rangle}$, which drives concentration and token clustering
\cite{geshkovski2025mathematical,chen2025quantitative}. The descending flow \eqref{e:wass grad flow} considered in the present paper has the opposite sign and therefore describes a dispersive counterpart of the attractive clustering dynamics arising in transformer dynamics. This relation places the same family of kernels at the intersection of transformer models, interacting-particle dynamics, and Wasserstein gradient flows; see \cite{geshkovski2025mathematical,rigollet2026mean,chen2025quantitative,chen2026propagation,bruno2025emergence,bruno2026multiscale}.

Another important family of zonal kernels includes the logarithmic, Coulomb, and Riesz kernels associated with the Euclidean chordal distance on the sphere. For any $\alpha \in ( -d,0) $, let
    \begin{align}\label{e:Riesz Euclidean kernel}
        K(x,y) =  |x-y|^\alpha, \quad x,y\in \bS^d,
    \end{align}
with the corresponding function $f(t) = 2^{\alpha/2} \cdot (1-t)^{\alpha/2}$. The associated Riesz energy is
\begin{align}\label{e:Riesz Euclidean energy}
\mathsf{E}[\mu]
=
\frac{1}{2}
\int_{\bS^d}
\int_{\bS^d}
|x-y|^\alpha
\,\mathrm{d}\mu(x)\,\mathrm{d}\mu(y).
\end{align}
The condition $\alpha>-d$ is the local integrability condition
for $K(x,y)$ with respect to spherical volume. When $\alpha \in (-d,0)$, the kernel $K$ is singular on the diagonal $x=y$, and hence does not satisfy the $C^2$-regularity requirement in \Cref{a:kernel}. Nevertheless, its Funk--Hecke coefficients are strictly positive. More precisely, in the spherical harmonic expansion \eqref{e:sphere kernel expansion}, its coefficients for \eqref{e:Riesz Euclidean kernel} are
\begin{align*}
a_\ell
=
c_{d,\alpha}
\frac{\bigl(-\tfrac{\alpha}{2}\bigr)_\ell}
{\bigl(d+\tfrac{\alpha}{2}\bigr)_\ell},
\ \forall \ell\geq0, \quad \text{where } c_{d,\alpha}
=
\Vol(\bS^d) \cdot 2^{d+\alpha-1}
\frac{
\Gamma\bigl(\tfrac{d+1}{2}\bigr)
\Gamma\bigl(\tfrac{d+\alpha}{2}\bigr)
}{
\sqrt{\pi} 
\Gamma\bigl(d+\tfrac{\alpha}{2}\bigr)
}
>0,
\end{align*}
and $(q)_\ell$ denotes the Pochhammer symbol. Since when $\alpha \in (-d,0)$, $-\frac{\alpha}{2}>0$ and $d+\frac{\alpha}{2}>0$, 
we have $a_\ell>0, \forall\,\ell\geq0$.

The singularity in \eqref{e:Riesz Euclidean kernel} when $\alpha \in(-d , 0 ) $ can be removed without destroying the positivity of the spherical harmonic coefficients. For any $\varepsilon>0$, consider the regularized Riesz kernel
\begin{align}\label{e:regularized Riesz Euclidean kernel}
K_{\varepsilon}(x,y)
=
\bigl(
\varepsilon^2+|x-y|^2
\bigr)^{\alpha/2}, \quad x , y \in \bS^{d} .
\end{align}
For $x,y \in \bS^{d}$, the Laplace-transform identity gives 
\begin{align*}
K_{\varepsilon}(x,y) =
\bigl(
\varepsilon^2+|x-y|^2
\bigr)^{\alpha/2} =
\frac1{\Gamma(-\alpha/2)}
\int_0^\infty
r^{-\alpha/2-1}
e^{-(\varepsilon^2+2)r}
e^{2r(x\cdot y)} 
\, \de r .
\end{align*}
For every $r>0$, the exponential kernel $e^{2r\,x\cdot y}$ has strictly positive spherical harmonic coefficients, as discussed above in \eqref{e:transformer kernel coefficients}. Therefore, $K_{\varepsilon}$ also has strictly positive coefficients. Thus, for every $\varepsilon>0$, the regularized Riesz kernel $K_{\varepsilon}$ satisfies all the assumptions in \Cref{a:kernel}. Moreover,
\begin{align*}
\lim_{\varepsilon \to 0} K_{\varepsilon}(x,y)
= 
K(x,y), \quad \forall x \neq y,
\end{align*}
and the corresponding regularized energies $\mathsf{E}[\cdot]$ converge to the Riesz energy \eqref{e:Riesz Euclidean energy} by the monotone convergence theorem.

These Riesz-type kernels are closely related to classical problems in potential theory, in which one seeks equilibrium distributions of particles interacting through a repulsive Riesz potential. Bilyk, Dai, and Matzke \cite{bilyk2019geodesic,bilyk2018stolarsky} study the analogous energy obtained by replacing the Euclidean chordal distance $|x-y|$ with the geodesic distance $\rho(x,y)=\arccos(x\cdot y)$.
They study the coefficients of the kernels
    \begin{align}\label{e:Riesz geodesic kernel}
        K(x,y) = \rho(x,y) ^{\alpha}
    \end{align}
for $\alpha\in(-d,0)$, as well as the associated discrete Riesz energies and their connections with spherical discrepancy and generalized Stolarsky principles.

\subsubsection{Inverse fractional Laplacian}\label{s:inverse laplacian kernel}
Let $M$ be a closed manifold of dimension $d \geq 1$ satisfying \eqref{a:normalized uniform measure}. For any $u \in L^p(M)$ with $p > 1$, we define its zero-mean projection by
\begin{align*}
\Pi_0 u \coloneqq u - \int_M u \de x ,
\end{align*}
and for any $s > 0$, we consider the operator
    \begin{align}\label{e:inverse laplacian operator}
         \cK[u] = (-\Delta)^{-s} (\Pi_0 u).
    \end{align}
The associated kernel $K$ is therefore
\begin{align}\label{e:inverse laplacian kernel}
K(x,y)=\sum_{n=1}^{\infty}\lambda_n^{-s}\phi_n(x)\phi_n(y).
\end{align}
Thus $a_0=0$ and $a_n=\lambda_n^{-s}>0$ for every $n\geq1$. When $s>d/2+1$, the Bernstein inequality in \Cref{lem:Bernstein inequalities} implies that $K\in C^2(M\times M)$, so \Cref{a:kernel} is satisfied.

For any $u \in \cP(M) \cap L^p(M)$, since $\int_M \cK[u]  \de x = 0$, the interaction energy becomes
    \begin{align*}
        \mathsf{E}[u] =\frac{1}{2} \int_M \Pi_0 u \cdot \cK[u] \, \de x = \frac{1}{2} \|\Pi_0 u \|_{\dot{H}^{-s}(M)} ^2 ,
    \end{align*}
where $\| \cdot \|_{\dot{H}^{-s}(M)} $ is the homogeneous Sobolev norm of order $-s$ on $M$. The recent work \cite{chizat2026quantitative} studies the Wasserstein gradient flow of the more general target-dependent functional on the flat torus $M = \TT^d$,
\begin{align}\label{e:laplacian inverse MMD}
\mathsf{E}_{v}[u]=\frac{1}{2}\|u-v\|_{\dot{H}^{-s}(\TT^d)}^2,
\end{align}
where $v\in\cP(\TT^d)$ is fixed. It establishes well-posedness in weak regularity classes for the full range $s\geq1$, including the singular regimes of the kernels \eqref{e:inverse laplacian kernel} outside the $C^2$-assumption of \Cref{a:kernel}. In particular, when $s = 1$, both \cite{chizat2026quantitative,de2026wasserstein} prove exponential decay of $\mathsf{E}_{v}[u]$ to $0$ for uniformly positive target $v$ along the Wasserstein gradient flow of the energy \eqref{e:laplacian inverse MMD}. For $s > 1$, \cite{chizat2026quantitative} proves local polynomial convergence under strict positivity and Sobolev regularity assumptions on both $u$ and $v$, provided that $\mathsf{E}_{v}[u]$ is initially small.

By contrast, \Cref{thm: intro global convergence} allows any nonnegative initial density $u_0\in L^p(M)$, $p>1$, and does not require closeness to the uniform measure initially. The $C^2$-assumption is used to obtain the requisite well-posedness and justify the entropy calculation, rather than in the underlying spectral coercivity mechanism. It is therefore natural to expect an extension of the entropy method to some of the singular inverse fractional Laplacian regimes once an appropriate weak-solution theory and entropy identity are available.


\section{Monotonicity of entropy: proof of \Cref{thm: intro global convergence} and \Cref{thm: intro mono entropy}}\label{s:global convergence entropy inequality}
In this section, we first deduce \Cref{thm: intro global convergence} from \Cref{thm: intro mono entropy}, and then prove the monotonicity of $\Ent(u_t)$ in \Cref{thm: intro mono entropy}.

\begin{proof}[Proof of \Cref{thm: intro global convergence}]
Let $ e(t) \coloneqq \mathsf{E}[\mu_t]-\mathsf{E}_{\min} \geq 0$. Integrating \eqref{e:mono entropy} in \Cref{thm: intro mono entropy} from $0$ to $T > 0$ gives
\begin{align*}
    2\lambda_1 \int_0^T e(t) \,\de t
    &\leq
    \Ent(u_0)-\Ent(u_T).
\end{align*}
The entropy $\Ent(u_T)$ is bounded below. More precisely, by Jensen's inequality,
\begin{align*}
    \Ent(u_T)
    =
    \int_M u_T\log u_T\,\de x
    \geq
    -\log (\Vol(M)),
\end{align*}
and equality holds only when $u_T = (\Vol(M))^{-1}$. Therefore, letting $T\to\infty$, we obtain
\begin{align}\label{e:energy-gap-integrable}
    2\lambda_1 \int_0^\infty e(t)\,\de t
    \leq \Ent(u_0) + \log (\Vol(M))<\infty.
\end{align}
This proves \eqref{e:energy integrability}. Then, since $e(t)$ is nonincreasing in $t$ by \eqref{e:mono energy}, we have that
\begin{align*}
    t e(t)
    &\leq
    2\int_{t/2}^{t} e(s)\,\de s \leq  2\int_{t/2}^{\infty} e(s)\,\de s\overset{\eqref{e:energy-gap-integrable}}{\longrightarrow} 0, \quad \text{ as } t \to \infty.
\end{align*}
This proves \eqref{e:energy convergence}.

It remains to identify the weak limits of $\mu_t$. Since $M$ is compact, Prokhorov's theorem implies that every sequence $\{\mu_{t_j}\}$ has a weakly convergent subsequence in $\cP(M)$. Let $\{t_j\}\to\infty$ be any sequence such that
\begin{align*}
    \mu_{t_j}\rightharpoonup \mu_\infty
\end{align*}
weakly for some $\mu_\infty \in \cP(M)$. Then, $\mu_{t_j} \otimes \mu_{t_j} \rightharpoonup \mu_\infty \otimes \mu_\infty$. Moreover, because the kernel $K$ is continuous on the compact space $M\times M$ under \Cref{a:kernel}, the map $\mu\mapsto \mathsf{E}[\mu]$ is continuous with respect to weak convergence on $\cP(M)$.
Combining this continuity with \eqref{e:energy convergence}, we find
\begin{align*}
    \mathsf{E}[\mu_\infty]
    &=
    \lim_{j\to\infty}\mathsf{E}[\mu_{t_j}]
    =
    \mathsf{E}_{\min}.
\end{align*}
Thus, every weak limit $\mu_\infty$ of $\mu_t$ is a minimizer of $\mathsf{E}$. 

In particular, if $\sigma\in\cP(M)$ is the unique minimizer of $\mathsf{E}$, then compactness of $\cP(M)$ and the preceding argument imply that
    \begin{align*}
        \mu_t \rightharpoonup \sigma,
    \end{align*}
because $\sigma$ is the only possible subsequential limit. Furthermore, on a compact metric space, $\bW_2$-convergence of probability measures is equivalent to weak convergence of probability measures (\cite[Theorem 6.9]{villani2009optimal}). Consequently,
\begin{align*}
    \lim_{t\to\infty}\bW_2(\mu_t,\sigma)=0.
\end{align*}

\end{proof}

We now prove \Cref{thm: intro mono entropy}. The proof is inspired by \cite[Theorem~2.1]{zimin2025learning}.

\begin{proof}[Proof of \Cref{thm: intro mono entropy}]
    We first prove that for any $t \geq 0$, $u_t \in L^p(M)$. Denote $X_t = -\gradW \mathsf{E}[\mu_t]$. Since $K\in C^2(M\times M)$ by \Cref{a:kernel} and $\mu_t$ is a probability measure, 
        \begin{align}\label{e:C^1 X_t bounded by C^2 K}
            \sup_{t\geq 0}\|X_t\|_{C^1(M)}
\leq C \cdot \|K\|_{C^2(M\times M)},
        \end{align}
    for a constant $ C = C(M) $, and $t \mapsto X_t$ is continuous. Therefore, $X_t$ generates a $C^1$ flow of diffeomorphisms $\phi_t$ satisfying
    \begin{align*}
        \partial_t\phi_t(x)=X_t(\phi_t(x)),
\qquad
\phi_0(x)=x.
    \end{align*}
Define the Jacobian
    \begin{align*}
        J_t(x)\coloneqq \det \left[ \nabla \phi_t(x) \right].
    \end{align*}
Liouville’s formula gives
    \begin{align}\label{e:Liouville's formula}
        J_t(x) = \exp \left( \int_0 ^t (\operatorname{div}X_r)(\phi_r(x)) \, \de r \right) .
    \end{align}
By \eqref{e:C^1 X_t bounded by C^2 K}, $\|\operatorname{div}X_t\|_{C^{0}(M)}
\leq C \cdot \|K\|_{C^2(M\times M)}$. Thus, $e^{-C \cdot \|K\|_{C^2} \cdot t}\le J_t(x)\le e^{C \cdot \|K\|_{C^2} \cdot t}$. Also, the continuity equation explicitly gives $\mu_t = (\phi_t)_{\#} \mu_0$ and thus
    \begin{align}\label{e:explicit characteristic formula}
        u_t(\phi_t(x))=u_0(x) \cdot J_t(x) ^{-1}.
    \end{align}
Thus,
    \begin{align*}
        \|u_t\|_{p} ^p =
\int_M |u_0(x)|^p J_t(x)^{1-p}\, \de x \le
e^{(p-1)C \|K\|_{C^2} \cdot t}\|u_0\|_{p}^p.
    \end{align*}
Hence $u_t\in L^p(M)$ for any $t \geq 0$.

Then, since $p>1$, there exists $C_p>0$ such that $z\log z\le C_p(1+z^p)$, while $z\log z\ge-e^{-1}$, for any $z \geq 0$. Thus, $\Ent(u_t)<\infty$ for any $ t \geq 0$. Also, by \eqref{e:explicit characteristic formula}, we can write
    \begin{align*}
        \Ent(u_t) =
\int_M
u_0(x)
\log\left(\frac{u_0(x)}{J_t(x)}\right)\, \de x=
\Ent(u_0)
-
\int_M u_0(x) \log J_t(x)\, \de x.
    \end{align*}
Using \eqref{e:Liouville's formula}, we obtain that for any $t \geq 0$,
    \begin{align*}
        \begin{split}
            &\Ent(u_t)-\Ent(u_0)
=
-\int_0^t\int_M
u_0(x)(\operatorname{div} X_r)(\phi_r(x))
\,\de x\, \de r\\
&=
-\int_0^t\int_M
u_r(y)(\operatorname{div}X_r)(y)
\,\de y\, \de r =
\int_0^t\int_M
u_r(y)\Delta \cK[u_r](y)
\, \de y\, \de r,
        \end{split}
    \end{align*}
where we used the change of variables $y = \phi_r(x)$ in the second equality and used \eqref{e:wass gradient} in the last equality. More generally, for any $0\leq t<s$,
    \begin{align}\label{e:difference entropy}
        \Ent(u_s) - \Ent(u_t) = \int_t ^s \int_M  u_r(y) \Delta \cK[u_r](y) \,\de y \de r.
    \end{align}
Because $r \mapsto \int_M  u_r(y) \Delta \cK[u_r](y) \,\de y$ is continuous, $\Ent(u_t)$ is thus differentiable in $t$ by \eqref{e:difference entropy} and satisfies
\begin{align}\label{e:derivative entropy}
    \frac{\de}{\de t}\Ent(u_t)
    =\int_M u_t(x)\,\Delta \cK[u_t](x)\,\de x.
\end{align}
For any fixed $t \geq 0$, let $\left\{u_t ^{(m)} \right\}_m \subseteq C^{\infty}(M) \cap \cP(M)$ such that $u_t ^{(m)} \to u_t$ in $L^p(M)$. For any fixed $m$, expand $u_t ^{(m)}$ using the orthonormal Laplacian eigenbasis:
\begin{align*}
    u_t ^{(m)} =\sum_{n=0}^{\infty} u_n ^{(m)} (t) \phi_n,
    \qquad
    u_n ^{(m)} (t) \coloneqq \int_M u^{(m)} _t\phi_n\,\de x.
\end{align*}
By the diagonal expansion of $K$ in \Cref{a:kernel}, we have
\begin{align*}
    \cK[u_t ^{(m)}](x)
    =
    \int_M K(x,y)u_t ^{(m)} (y)\,\de y
    =
    \sum_{n=0}^{\infty}a_n u_n ^{(m)} (t)\phi_n(x),
\end{align*}
and thus
\begin{align}\label{e:Lap K[u_t] expansion}
    \Delta \cK[u_t ^{(m)}](x)
    =
    -\sum_{n=1}^{\infty}\lambda_n a_n u_n ^{(m)} (t)\phi_n(x),
\end{align}
where the $n=0$ term disappears because $\lambda_0=0$. Hence,
\begin{align*}
    \int_M u_t ^{(m)}(x)\,\Delta \cK[u_t ^{(m)}](x)\,\de x
    &=
    \int_M
    \left(\sum_{k=0}^{\infty} u_k ^{(m)} (t)\phi_k(x)\right)
    \left(-\sum_{n=1}^{\infty}\lambda_n a_n  u_n ^{(m)} (t)\phi_n(x)\right)
    \, \de x \\
    &=
    -\sum_{n=1}^{\infty}\lambda_n a_n \left(u_n ^{(m)}(t) \right) ^2,
\end{align*}
where we used the fact that $\int_M \phi_m \phi_n \de x = \delta_{mn}$.
Since $a_n\geq 0$ and $\lambda_n\geq \lambda_1$ for every $n\geq 1$ by \Cref{a:kernel}, we obtain
\begin{align*}
    \int_M u_t ^{(m)}(x)\,\Delta \cK[u_t ^{(m)}](x)\,\de x
    \leq
    -\lambda_1\sum_{n=1}^{\infty}a_n \left(u_n ^{(m)}(t) \right) ^2 = -2\lambda_1 \left( \mathsf E[u_t ^{(m)}]-\mathsf E_{\min} \right),
\end{align*}
where the last equality follows from \eqref{e:energy expansion}.
Because $u_t ^{(m)} \to u_t$ in $L^p(M)$, passing to the limit $m\to\infty$ in this inequality and using \eqref{e:derivative entropy}, we obtain
    \begin{align*}
        \begin{split}
            \frac{\de}{\de t}\Ent(u_t)
    &=\int_M u_t(x)\,\Delta \cK[u_t](x)\,\de x = \lim_{m \to \infty} \int_M u_t ^{(m)}(x)\,\Delta \cK[u_t ^{(m)}](x)\,\de x 
    \\  &\leq \lim_{m \to \infty} -2\lambda_1 \left( \mathsf E[u_t ^{(m)}]-\mathsf E_{\min} \right) =  -2\lambda_1 \left( \mathsf E[u_t]-\mathsf E_{\min} \right) .
        \end{split}
    \end{align*}
This proves \eqref{e:mono entropy}.
\end{proof}

\section{Linearized flows around the uniform measure: proof of \Cref{thm:linear-sharp-delta subseq intro} and \Cref{thm:linear-sharp-delta intro}}\label{s:linearization flow}

In this section, we study the linearization of \eqref{e:wass grad flow} around the uniform measure $\sigma =  \de x$, which under \Cref{a:kernel} is a minimizer of $\mathsf E$ over $\cP(M)$. To derive the linearized equation, consider a solution with density
\begin{align}\label{e:linearization around sigma}
    u_t^\varepsilon=1+\varepsilon f_t^\varepsilon,
    \qquad
    \int_M f_t^\varepsilon\,\de x=0.
\end{align}
Since $\cK[1]=a_0$ is constant, $\nabla\cK[1]=0$. Substituting \eqref{e:linearization around sigma} into \eqref{e:wass grad flow} gives the exact perturbation equation
\begin{align}\label{e:wass grad flow perturbation linear flow section}
    \partial_t f_t^\varepsilon-\Delta\cK[f_t^\varepsilon]
    =
    \varepsilon\,\dive\left(f_t^\varepsilon\nabla\cK[f_t^\varepsilon]\right).
\end{align}
Discarding the quadratic term, or equivalently setting $\varepsilon=0$, yields the linearized equation
\begin{align}\label{e:linearized flow}
    \partial_t\overline f_t-\Delta\cK[\overline f_t]=0,
    \qquad
    \overline f_0=f_0,
    \qquad
    \int_M\overline f_t\,\de x=0.
\end{align}
Equivalently, $\overline{f}_t$ satisfies
\begin{align}\label{e:linearized semigroup}
    \overline{f}_t=e^{t\Delta \cK }f_0.
\end{align}
Under \Cref{a:kernel}, we can use \eqref{e:Lap K[u_t] expansion} to further write $\overline{f}_t$ in the following form: $\forall t \geq 0$,
    \begin{align}\label{e:linearized basis expansion}
        \overline{f}_t = \sum_{ n = 1} ^{\infty} e^{-a_n \lambda_n t} b_n \phi_n, \quad \text{ with } b_n \coloneqq \int_M f_0 \phi_n \,\de x.
    \end{align}
Thus, the decay of the linearized flow is determined by the spectral rates $\{a_n\lambda_n\}_{n\geq1}$.

The corresponding linearized energy is the quadratic term in the energy gap. Indeed, for every mean-zero function $f$,
\begin{align}\label{e:energy linearization around sigma}
    \mathsf E[1+\varepsilon f]-\mathsf E_{\min}
    =
    \varepsilon^2\mathsf E[f].
\end{align}
We therefore set the linearized energy as
\begin{align}\label{e:linearized energy}
    \mathsf{E}_{\mathrm{linear}}[\mu_t^\varepsilon]
    \coloneqq
    \mathsf E[\overline f_t].
\end{align}
Using the basis expansion \eqref{e:linearized basis expansion},
    \begin{align}\label{e:linearized energy basis expansion}
        \mathsf{E}[\overline{f}_t] = \frac{1}{2}\sum_{ n = 1} ^{\infty} a_n b_n ^2 e^{-2 a_n \lambda_n t} .
    \end{align}
We next show that one can choose a kernel $K$ satisfying \Cref{a:kernel} and an initial datum $f_0\in H^{-1}(M)$ for which the linearized energy decays at a rate arbitrarily close to $t^{-1}$, which proves \Cref{thm:linear-sharp-delta intro}.

\begin{theorem}[Sharp $o(t^{-1})$-rate for the linearized flow]\label{thm:linear-sharp-delta}
There exists a kernel $K \in C^{\infty}(M \times M)$ satisfying \Cref{a:kernel}, such that for any $\delta>0$, there is a zero-mean function $f_0\in C^\infty(M)$, such that the solution $\overline{f}_t$ to the linearized equation \eqref{e:linearized flow} satisfies that
\begin{align*}
    e^{-5} \cdot 
    t^{-1-\delta} \leq \mathsf E[\overline{f}_t]
    \leq C_{\delta} \cdot 
    t^{-1-\delta} , \quad \forall t \geq 2.
\end{align*}
Here, $C_{\delta} >0$ is a constant depending only on $\delta$.
\end{theorem}

\begin{proof}
Recall the orthonormal basis of Laplacian eigenfunctions $\{\phi_n\}_{n\geq0}$ from \Cref{a:kernel}. Their eigenvalues $0=\lambda_0<\lambda_1 \leq \lambda_2 \leq \cdots $ satisfy the following consequence of Weyl’s law: 
\begin{align}\label{e:Weyl law}
    C^{-1} n^{2/d} \leq \lambda_n \leq C n^{2/d}, \quad n \geq 1,
\end{align}
where $C = C(M) >1$ is a constant and $d=\dim M$.

We first construct the kernel $K$. For any $n \geq 1 $, we set
\begin{align*}
    a_0 = 0, \quad a_n= 2^{-n}\cdot \lambda_{n} ^{-1} , \quad n \geq 1.
\end{align*}
Combining with \eqref{e:Weyl law}, the sequence $\{a_n\}_{n \geq 0}$ decays faster than any negative power of $n$. Therefore the series defining $K$ in \Cref{a:kernel}, i.e.,
    \begin{align*}
            K(x,y) = \sum_{n=1} ^{\infty} a_n \phi_n(x) \phi_n(y),
    \end{align*}
converges in $C^\infty(M\times M)$ by the Bernstein inequalities \Cref{lem:Bernstein inequalities}. By this construction, the kernel $K$ satisfies \Cref{a:kernel}.

We next choose the coefficients of $f_0$. For any $n \geq 1 $, we choose $b_n > 0$ such that
    \begin{align}\label{e:linearized bn choice}
    a_n b_n ^2
    = 2^{-n(1+\delta)},
\end{align}
or equivalently, 
    \begin{align*}
        b_n = \left(a_n 2^{n(1+  \delta)} \right)^{-1/2} = 2^{-n\delta /2} \lambda_n ^{1/2}.
    \end{align*}
Similarly, the sequence $\{b_n\}_{n \geq 0}$ decays faster than any negative power of $n$, and thus
\begin{align*}
    f_0=\sum_{n=1}^\infty b_n \phi_n
\end{align*}
converges in $C^\infty(M)$ and has zero mean.

We now estimate $\mathsf{E}[\overline{f}_t]$. By \eqref{e:linearized energy basis expansion} and \eqref{e:linearized bn choice},
\begin{align*}
    2\mathsf E[\overline{f}_t]
    =
    \sum_{n=1}^\infty
    2^{-n(1+\delta)}
    e^{-2\cdot 2^{-n}t}.
\end{align*}
For any $t\geq 2$, let $n_t \coloneqq \lfloor \log_2 (t) \rfloor \in \bZ_+$, and thus $2^{n_t} \leq t \leq 2 \cdot 2^{n_t}$. We see that
\begin{align*}
    2\mathsf E[\overline{f}_t]
    &\geq
    2^{-n_t(1+\delta)}
    e^{-2\cdot 2^{-n_t}t} \geq t^{-(1+\delta)} e^{-4},
\end{align*}
which gives the lower bound in \Cref{thm:linear-sharp-delta}. For the upper bound, we split $\mathsf{E}[\overline{f}_t]$ into 
\begin{align*}
    2\mathsf{E}[\overline{f}_t]
    =
    \sum_{1\leq n<n_t} 2^{-n(1+\delta)} e^{-2\cdot 2^{-n}t}
    + \sum_{n\geq n_t} 2^{-n(1+\delta)} e^{-2\cdot 2^{-n}t}
     \eqqcolon I_1+I_2
\end{align*}
For $I_1$,
\begin{align*}
    \begin{split}
        I_1 &= 2^{-n_t(1+\delta)}\sum_{1\leq n<n_t} 2^{(n_t-n)(1+\delta)} e^{-2\cdot 2^{-n}t}
        \leq t^{-(1+\delta)} \cdot 2^{1+\delta} \sum_{1\leq n<n_t} 2^{(n_t-n)(1+\delta)} e^{-2\cdot 2^{(n_t-n)}}
        \\  &= t^{-(1+\delta)} \cdot 2^{1+\delta} \sum_{p=1} ^{n_t - 1} 2^{p(1+\delta)} e^{-2 \cdot 2^p} 
        \leq  t^{-(1+\delta)} \cdot 2^{1+\delta} \sum_{k=1} ^{\infty} k^{1+\delta} e^{-2k} \leq t^{-(1+\delta)} \cdot C_{\delta} .
    \end{split}
\end{align*}
For $I_2$,
\begin{align*}
    I_2
    \leq
    \sum_{n\geq n_t} 2^{-n(1+\delta)}
    =
    \frac{2^{-n_t (1+\delta)}}{1-2^{-(1+\delta)}} \leq \frac{2^{(1+\delta)}}{1-2^{-(1+\delta)}} t^{-(1+\delta)}.
\end{align*}
Combining the estimates for $I_1$ and $I_2$, we obtain the upper bound in \Cref{thm:linear-sharp-delta}.

\end{proof}

We next consider an arbitrary kernel satisfying \Cref{a:kernel}. If only finitely many $a_n$ are positive, then \eqref{e:linearized energy basis expansion} is a finite sum of exponentially decaying terms. Slow algebraic decay can therefore occur only when the kernel has infinitely many positive modes.

\begin{theorem}[Sharp $o(t^{-1})$-rate for the linearized flow of arbitrary kernels]\label{thm:linear-sharp-delta subseq}
For any $\delta>0$ and any kernel $K$ satisfying \Cref{a:kernel} with infinitely many positive $a_n$, there exist a $f_0\in H^{-1}(M)$ with zero mean and a sequence of time $\{t_k\}_{k \geq 1} \to \infty$, such that the solution $\overline{f}_t \in H^{-1}(M) $ to the linearized equation \eqref{e:linearized flow} satisfies that
\begin{align*}
    e ^{-3} \cdot 
    t_k ^{-1-\delta} \leq \mathsf E[\overline{f}_{t_k}]
    \leq 2 \cdot 
    t_k ^{-1-\delta} , \quad \forall k \geq 1.
\end{align*}
\end{theorem}
\begin{proof}
Since $K\in C^2(M\times M)$, we have $\Delta_xK\in L^2(M\times M)$. Therefore $\sum_{n \geq 1}(\lambda_n a_n)^2 <\infty$. In particular, $\lim_{n \to \infty} \lambda_n a_n=0$. Then, we can inductively choose a subsequence of indices, i.e., $\{n_k\}_{k \geq 1} \subseteq \bZ_+$, such that $a_{n_k} > 0 $ and 
    \begin{align}\label{e:kernel subindices 1}
           (\lambda_{n_{k+1}} a_{n_{k+1}})^{1+\delta} \leq \frac{1}{4} (\lambda_{n_k} a_{n_k})^{1+\delta}
    \end{align}
and 
    \begin{align}\label{e:kernel subindices 2}
                \sum_{p<k}
        (\lambda_{n_p} a_{n_p})^{1+\delta}
        \exp\left(-2\frac{\lambda_{n_p} a_{n_p}}{\lambda_{n_k} a_{n_k}}\right)
        \leq (\lambda_{n_k} a_{n_k})^{1+\delta}.
    \end{align}
\eqref{e:kernel subindices 1} follows easily as $\lim_{n \to \infty} \lambda_n a_n=0$. For \eqref{e:kernel subindices 2}, we need to use the fact that the function 
    \begin{align*}
        h(s) \coloneqq \sum_{p<k}
        (\lambda_{n_p} a_{n_p})^{1+\delta}
        \exp\left(-2\frac{\lambda_{n_p} a_{n_p}}{s}\right)
    \end{align*}
tends to $0$ exponentially fast as $s \to 0^+$ and thus faster than $s^{1+\delta}$.

Now, define $f_0$ by setting $b_n=0$ unless $n=n_k$. When $n = n_k$, choose $b_{n_k}$ so that
\begin{align*}
     a_{n_k} b_{n_k} ^2
    =
    (\lambda_{n_k} a_{n_k})^{1+\delta}.
\end{align*}
First,
\begin{align*}
    f_0
    =
    \sum_{k= 1} ^{\infty} b_{n_k}\phi_{n_k}
\end{align*}
has zero mean because every $n_k \geq 1$. We next check that $f_0\in  H^{-1}(M)$. Indeed, by \eqref{e:kernel subindices 1}, the sequence $\{\lambda_{n_k} a_{n_k}\}_{k \geq 1}$ decays geometrically, and thus for some constant $C_{\delta}$ only depending on $\delta$, 
\begin{align*}
    \|f_0\|_{ H^{-1}}^2
    =
    \sum_{k= 1} ^{\infty}
    \lambda_{n_k}^{-1} b_{n_k} ^2 =
    \sum_{k= 1} ^{\infty}(\lambda_{n_k} a_{n_k})^{\delta} \leq (\lambda_{n_1} a_{n_1})^{\delta} C_{\delta} < \infty.
\end{align*}

For any $k \geq 1$, we then set
\begin{align*}
    t_k\coloneqq (\lambda_{n_k} a_{n_k})^{-1}.
\end{align*}
At time $t_k$, the energy in \eqref{e:linearized energy} becomes
\begin{align*}
    2\mathsf E[\overline{f}_{t_k}] =
    \sum_{p= 1} ^{\infty}
   (\lambda_{n_p} a_{n_p})^{1+\delta}
    \exp\left(-2\frac{\lambda_{n_p} a_{n_p}}{\lambda_{n_k} a_{n_k}}\right).
\end{align*}
For the lower bound in \Cref{thm:linear-sharp-delta subseq}, 
\begin{align*}
    2\mathsf E[\overline{f}_{t_k}] 
    \geq (\lambda_{n_k} a_{n_k})^{1+\delta} e^{-2} = t_k ^{-(1+\delta)} e^{-2}.
\end{align*}
For the upper bound in \Cref{thm:linear-sharp-delta subseq}, we split the sum into two parts:
\begin{align*}
    2\mathsf E[\overline{f}_{t_k}]
    =
    \left(\sum_{p< k} +  \sum_{p \geq k} \right)
   (\lambda_{n_p} a_{n_p})^{1+\delta}
    \exp\left(-2\frac{\lambda_{n_p} a_{n_p}}{\lambda_{n_k} a_{n_k}}\right) \eqqcolon I_1+I_2.
\end{align*}
For $I_1$, we make use of the choice of $\{n_k\}$, as $I_1$ is exactly the term on the left-hand side of \eqref{e:kernel subindices 2}. Hence, 
\begin{align*}
    I_1 \leq (\lambda_{n_k} a_{n_k})^{1+\delta} = t_k ^{-(1+\delta)}
\end{align*}
For $I_2$, using \eqref{e:kernel subindices 1} and the fact that $\exp\left(-2\frac{\lambda_{n_p} a_{n_p}}{\lambda_{n_k} a_{n_k}}\right) \leq 1$ when $p \geq k$, we get
\begin{align*}
    I_2 \leq  \sum_{p \geq k} 
   (\lambda_{n_p} a_{n_p})^{1+\delta} \leq (\lambda_{n_k} a_{n_k})^{1+\delta} \sum_{q=0} ^{\infty} 4^{-q} = \frac{4}{3} t_k ^{-(1+\delta)}.
\end{align*}
Combining the estimates for $I_1$ and $I_2$, we obtain the upper bound in \Cref{thm:linear-sharp-delta subseq}.
\end{proof}

The nonlinear $o(t^{-1})$ estimate in \Cref{thm: intro global convergence} is established for initial densities in $L^p(M)$ with $p>1$, whereas the initial datum constructed in \Cref{thm:linear-sharp-delta subseq} generally belongs only to $H^{-1}(M)$. It is therefore natural to ask whether there exists $f_0\in H^{-1}(M)$ for which the corresponding solution $\overline{f}_t$ of \eqref{e:linearized flow} fails to satisfy the $o(t^{-1})$ rate. The next proposition gives a negative answer. This conclusion can also be deduced from the general Hilbert-space gradient flow result in \cite{siegel2024qualitative}, since \eqref{e:linearized flow} can also be regarded as the $H^{-1}(M)$-gradient flow of the convex quadratic functional $\mathsf E$. We include a direct spectral proof adapted to the present setting.

\begin{proposition}[$o(t^{-1})$-rate for $H^{-1}$ initial data]\label{prop:H^-1 initial o(t^-1) rate}
    For any kernel $K$ satisfying \Cref{a:kernel} and any $f_0 \in H^{-1}(M)$ with zero mean, the solution $\overline{f}_t$ to the linearized equation \eqref{e:linearized flow} satisfies that, as $t \to \infty$, 
        \begin{align*}
            \mathsf{E}[\overline{f}_t] = o(t^{-1}).
        \end{align*}
\end{proposition}
\begin{proof}
    Using \eqref{e:linearized energy}, we see that
        \begin{align*}
            \int_0 ^{\infty} \mathsf{E}[\overline{f}_s] \,\de s = \frac{1}{2}\sum_{ n = 1} ^{\infty} a_n b_n ^2 \int_0 ^{\infty} e^{-2 a_n \lambda_n s} \,\de s  = \frac{1}{4} \sum_{ a_n > 0} \lambda_n ^{-1} b_n ^2 \leq \frac{1}{4} \|f_0\|_{ H^{-1}}^2.
        \end{align*}
    Since $\mathsf{E}[\overline{f}_t]$ is nonincreasing in $t$ by \eqref{e:linearized energy basis expansion}, we have that
        \begin{align*}
            t \mathsf{E}[\overline{f}_t] \leq 2 \int_{t/2} ^{t} \mathsf{E}[\overline{f}_s] \,\de s \leq 2 \int_{t/2} ^{\infty} \mathsf{E}[\overline{f}_s] \,\de s \to 0, \quad \text{ as } t \to \infty.
        \end{align*}
\end{proof}

\Cref{prop:H^-1 initial o(t^-1) rate} gives the universal $o(t^{-1})$ upper bound for every kernel $K$ satisfying \Cref{a:kernel} and every mean-zero $f_0 \in H^{-1}(M)$. For a general kernel, however, the two-sided estimate in \Cref{thm:linear-sharp-delta subseq} cannot be extended from a sequence of times to all sufficiently large times.
\begin{theorem}[Failure of uniform-in-time polynomial asymptotics]\label{thm:lacunary no uniform rate}
There exists a kernel
$K\in C^\infty(M\times M)$ satisfying \Cref{a:kernel} such that, for every
mean-zero $f_0\in H^{-1}(M)$ and every $\beta>0$, the solution $\overline{f}_t$ of
\eqref{e:linearized flow} satisfies:
\begin{align*}\limsup_{t\to\infty} \, t^{\beta} \cdot \mathsf E[\overline{f}_t]<\infty
    \quad\Longrightarrow\quad
    \liminf_{t\to\infty} \, t^{5\beta/4} \cdot \mathsf E[\overline{f}_t]=0.
\end{align*}
\end{theorem}
\begin{proof}
Let the spectral coefficients of $K$ be
\begin{align}\label{e:universal lacunary kernel}
    a_0=0,
    \qquad
    a_n=\lambda_n^{-1}e^{-2^n},
    \quad n\geq1.
\end{align}
Equivalently,
\begin{align*}
    K(x,y)
    =
    \sum_{n=1}^{\infty}
    \lambda_n^{-1}e^{-2^n}\phi_n(x)\phi_n(y).
\end{align*}
By an argument similar to that in the proof of \Cref{thm:linear-sharp-delta}, $K\in C^\infty(M\times M)$ and satisfies
\Cref{a:kernel}.

Let
    \begin{align}\label{e:t_n time}
        t_n=e^{2^n}, \, n \geq 1 .
    \end{align}
Then $t_{n+1}=t_n^2$. For a mean-zero $f_0\in H^{-1}(M)$ and the corresponding solution of \eqref{e:linearized flow}, write
\begin{align*}
    f_0=\sum_{n=1}^{\infty}b_n\phi_n , \qquad \overline{f}_t
    =
    \sum_{n=1}^{\infty}
    b_ne^{-t_n ^{-1} t}\phi_n.
\end{align*}
Consequently,
\begin{align}\label{e:universal lacunary energy}
    \mathsf E[\overline{f}_t]
    =
    \frac{1}{2} \sum_{n=1}^{\infty} c_n  e^{-2t_n^{-1}t}, \quad \text{ where } c_n = a_n b_n^2.
\end{align}
Fix a $\beta >0$ and suppose that $\limsup_{t\to\infty}t^{\beta} \mathsf E[\overline{f}_t]<\infty$. 
Then there exists a constant $C>1$ such that
\begin{align}\label{e:lacunary assumed upper}
    \mathsf E[\overline{f}_t]\leq C(t+1)^{-\beta},
    \qquad
    \forall t\geq 0.
\end{align}
Evaluating \eqref{e:universal lacunary energy} at
$t=t_n$ and keeping only the $n$-th term, we obtain
\begin{align}\label{e:lacunary coefficient upper}
    \frac{1}{2} c_ne^{-2}
    \leq
    \mathsf E[\overline{f}_{t_n}]
    \leq
    C(t_n+1)^{-\beta}
    \leq 
    C t_n ^{-\beta} , \quad \forall n \geq 1 .
\end{align}

We now consider the intermediate times $s_n = \sqrt{t_n t_{n+1}} = t_n^{3/2}$. We claim that
\begin{align}\label{e:lacunary universal intermediate decay}
    \lim_{n\to\infty}s_n^{5\beta/4}\mathsf E[\overline{f}_{s_n}]=0.
\end{align}
Indeed, by \eqref{e:universal lacunary energy},
\begin{align}\label{e:lacunary universal normalized energy}
    2s_n^{5\beta/4}\mathsf E[\overline{f}_{s_n}]
    =
    \sum_{j=1}^{\infty}
    c_j s_n^{5\beta/4} e^{-2t_j ^{-1} s_n} .
\end{align}
We split the sum in \eqref{e:lacunary universal normalized energy} into $j\leq n$ and $j\geq n+1$. For $j \leq n$, because the function $s \mapsto s^{\beta} e^{-2s}$ is decreasing for $s \geq \beta/2$ and $t_j^{-1} s_n \geq t_n^{-1} s_n = t_n^{1/2} \to \infty$, 
it follows from \eqref{e:lacunary coefficient upper} that, for all
sufficiently large $n$,
    \begin{align}\label{e:lacunary intermediate energy j leq n}
        \begin{split}
            &\sum_{j \leq n}
    c_j s_n^{5\beta/4} e^{-2t_j ^{-1} s_n} 
    \leq 
    e^3 C \cdot s_n^{\beta/4}\sum_{j \leq n}
     (t_j^{-1} s_n)^{\beta} e^{-2t_j ^{-1} s_n} 
     \\ &\leq e^3 C \cdot s_n^{\beta/4} \cdot n (t_n^{-1} s_n)^{\beta} e^{-2 t_n ^{-1} s_n} =
    e^3 C \cdot n  \cdot t_n ^{7\beta/8} e^{-2 t_n^{1/2} },
        \end{split}
\end{align}
which tends to $0$ as $n \to \infty$ by the choice of $t_n$ in \eqref{e:t_n time}. 
For $j \geq n+1$, by \eqref{e:lacunary coefficient upper} and
the fact that $e^{-2t_j ^{-1} s_n}\leq1$,
\begin{align*}
    \sum_{j\geq n+1}
    c_js_n^{5\beta/4} e^{-2t_j ^{-1} s_n}
    \leq
    e^3 C \cdot s_n^{5\beta/4} \sum_{j\geq n+1}t_j^{-\beta}.
\end{align*}
By \eqref{e:t_n time}, we have that $t_{j+1} = t_{j}^2 \geq  t_j t_{n+1} $.
Therefore, 
\begin{align}\label{e:lacunary intermediate energy j geq n+1}
    \sum_{j\geq n+1}
    c_js_n^{5\beta/4} e^{-2t_j ^{-1} s_n}
    \leq
    e^3 C \cdot \frac{(t_{n+1}^{-1}s_n ^{5/4})^{\beta}}{1-t_{n+1} ^{-\beta}} \leq e^3 C \cdot \frac{(t_{n+1}^{-1}s_n ^{5/4})^{\beta}}{1-e ^{-\beta}} =  e^3 C \cdot \frac{(t_n)^{-\beta/8}}{1-e ^{-\beta}},
\end{align}
which tends to $0$ as $n \to \infty$. 
Here the second inequality follows from $t_{n+1} \geq e$. 

\eqref{e:lacunary intermediate energy j leq n} and \eqref{e:lacunary intermediate energy j geq n+1} together prove
\eqref{e:lacunary universal intermediate decay}. Consequently, $\liminf_{t\to\infty}t^{5\beta/4}\mathsf E[\overline{f}_t]=0$.
\end{proof}


For the inverse fractional Laplacian kernels $ \cK=(-\Delta)^{-s}\circ\Pi_0$ introduced in \Cref{s:inverse laplacian kernel}, greater regularity of the initial datum leads to faster decay of $\mathsf E[\overline f_t]$; see \cite[Theorem~1.4]{chizat2026quantitative}. By contrast, a mean-zero signed measure outside $H^{-1}(M)$ can yield decay strictly slower than $t^{-1}$. We demonstrate this for $s > d/2$ and $d \geq 3$ by taking $f_0=\delta_{x_0}-\sigma$, where $\delta_{x_0}$ is the Dirac measure at an arbitrary point $x_0 \in M$.

\begin{proposition}\label{prop:linearized equation slower t^-1 rate}
    For any $d \geq 2$ and any $x_0 \in M$, the zero-mean signed measure $f_0 =\delta_{x_0} - \sigma \notin H^{-1}(M)$. Moreover, if $\cK = (-\Delta)^{-s} \circ \Pi_0$ for some $s>d/2$ as in \Cref{s:inverse laplacian kernel}, the solution $\overline{f}_t$ to the linearized equation \eqref{e:linearized flow} satisfies that as $t \to \infty$,
        \begin{align*}
    C_{d,s}^{-1} \cdot 
    t^{-\beta} \leq \mathsf E[\overline{f}_t]
    \leq C_{d,s} \cdot 
    t^{-\beta} .
\end{align*}
Here, $C_{d,s} >1$ is a constant depending only on $d$ and $s$, and $\beta \coloneqq \frac{s-d/2}{s-1}$, which is less than $1$ if $d \geq 3$.
\end{proposition}
\begin{proof}
By \Cref{s:inverse laplacian kernel}, $a_n=\lambda_n^{-s}$ for $n\geq1$. For the signed measure $f_0=\delta_{x_0}-\sigma$, we have that $b_n= \int_M f_0 \phi_n \,\de x = \phi_n(x_0)$ for $n \geq 1$.
According to \eqref{e:linearized basis expansion} and \eqref{e:linearized energy basis expansion},
\begin{align*}
    \overline{f}_t
    =
    \sum_{n=1}^{\infty}
    e^{-\lambda_n ^{1-s} t}\phi_n(x_0)\phi_n,
\end{align*}
and 
\begin{align}\label{e:signed-measure-energy}
    \mathsf E[\overline{f}_t]
    =
    \frac{1}{2}
    \sum_{n=1}^{\infty}
    \lambda_n^{-s}
     |\phi_n  (x_0) |^2
    e^{-2\lambda_n^{1-s} t}.
\end{align}

To proceed, we need the local Weyl law (\cite{hormander1968spectral,shubin1987pseudodifferential,jakobson2007estimates}) for the spectral function. Namely, as $\Lambda \to\infty$,
\begin{align}\label{e:local-weyl-law}
    \sum_{\lambda_n\leq \Lambda}|\phi_n(x)|^2
    =
    C_d \Lambda^{d/2}
    +
    O(\Lambda^{(d-1)/2}),
\end{align}
uniformly for $x\in M$, for some constant $C_d>0$ only depending on the dimension $d$. From \eqref{e:local-weyl-law}, there are constants $C_d >1$, $\Lambda_M >1$, such that when $\Lambda > \Lambda_M$,
\begin{align}\label{e:dyadic-local-weyl}
    C_d ^{-1} \Lambda^{d/2}
    \leq
    \sum_{\Lambda \leq \lambda_n\leq 2 \Lambda}|\phi_n(x_0)|^2
    \le
    C_d  \Lambda ^{d/2}.
\end{align}
Clearly, $f_0 \notin H^{-1}(M)$, because by \eqref{e:dyadic-local-weyl},
    \begin{align*}
        \begin{split}
               &\|f_0\|_{H^{-1}} ^2 =\sum_{n=1} ^{\infty} \lambda_n^{-1}
     |\phi_n  (x_0) |^2 \geq \sum_{k =0} ^{\infty} \sum_{\lambda_n  \in [2^{k}\Lambda_M , 2^{k+1}\Lambda_M)} \lambda_n^{-1}
     |\phi_n  (x_0) |^2 
     \\ &\geq \sum_{k =0} ^{\infty} 2^{-k-1} \Lambda_M ^{-1} C_d^{-1} \left( 2^{k}\Lambda_M\right)^{d/2} = \Lambda_M ^{d/2 -1} (2C_d) ^{-1} \sum_{k=0} ^{\infty} \left( 2^{d/2-1} \right)^k = \infty,
        \end{split}
    \end{align*}
when $d \geq 2$.

We now estimate $\mathsf E[\overline{f}_t]$ from \eqref{e:signed-measure-energy}. For any $t > \Lambda_M ^{(s-1)}$, we set 
    \begin{align*}
        \Lambda_t \coloneqq t^{1/(s-1)},
    \end{align*}
so that $\Lambda_t > \Lambda_M$. For the lower bound in \Cref{prop:linearized equation slower t^-1 rate},  we see that when $\lambda_n \in [\Lambda_t, 2\Lambda_t] $, $t \lambda_n ^{1-s} \leq t \Lambda_t ^{1-s} = 1$. Thus,
\begin{align*}
\begin{aligned}
    2\mathsf E[\overline{f}_t]
    &\geq 
    \sum_{\Lambda_t \leq \lambda_n\leq 2\Lambda_t}
    \lambda_n^{-s}
    |\phi_n(x_0)|^2
    e^{-2t\lambda_n^{1-s}}   \geq
    2^{-s} \Lambda_t ^{-s} e^{-2}
    \sum_{\Lambda_t \leq \lambda_n\leq 2\Lambda_t}
    |\phi_n(x_0)|^2             \\
    &\overset{\eqref{e:dyadic-local-weyl}}{\geq} C_{d,s} ^{-1}
     \Lambda_t ^{d/2-s}      = C_{d,s} ^{-1} t^{-\beta}.
\end{aligned}
\end{align*}
For the upper bound in \Cref{prop:linearized equation slower t^-1 rate}, we again split the sum \eqref{e:signed-measure-energy} into two parts:
    \begin{align*}
        2\mathsf E[\overline{f}_t] = \left(\sum_{\lambda_n < \Lambda_t} +  \sum_{\lambda_n \geq \Lambda_t} \right) \lambda_n^{-s}
     |\phi_n  (x_0) |^2
    e^{-2\lambda_n^{1-s} t} \eqqcolon I_1 + I_2.
    \end{align*}
In $I_1$, those terms $\lambda_n \leq \Lambda_M$ decay exponentially fast with a rate at least $e^{-2\Lambda_M ^{1-s}t}$, which are negligible. Thus,
    \begin{align*}
        \begin{split}
            &I_1 \leq \sum_{p=0} ^{\infty} \sum_{\lambda_n \in [2^{-p-1}\Lambda_t, \ 2^{-p}\Lambda_t]} \lambda_n^{-s}
     |\phi_n  (x_0) |^2
    e^{-2\lambda_n^{1-s} t} 
    \\ &\overset{\eqref{e:dyadic-local-weyl}}{\leq} C_M e^{-2\Lambda_M ^{1-s}t}+  C_{d,s} \Lambda_t ^{-s}\sum_{p=0} ^{\infty} 2^{p s} e^{-2 \cdot 2^{p(s-1)}} \left( 2^{-p}\Lambda_t\right)^{d/2}
    \\  &=C_M e^{-2\Lambda_M ^{1-s}t}+  C_{d,s} \Lambda_t ^{d/2-s}\sum_{p=0} ^{\infty} 2^{p (s-d/2)} e^{-2 \cdot 2^{p(s-1)}} 
    \\  &\leq C_M e^{-2\Lambda_M ^{1-s}t}+  C_{d,s} \Lambda_t ^{d/2-s}\sum_{k=0} ^{\infty} k^{(s-d/2)} e^{-2 \cdot k^{(s-1)}}
    \\ &= C_M e^{-2\Lambda_M ^{1-s}t}+  C_{d,s} t^{-\beta},
        \end{split}
    \end{align*}
where in the last equality, we used the fact that when $s >d/2 \geq 1$, the series $\sum_{k=0} ^{\infty} k^{(s-d/2)} e^{-2 \cdot k^{(s-1)}}$ is convergent and equals a constant $C_{d,s}$ only depending on $d,s$. For $I_2$, we notice the fact that $e^{-2\lambda_n^{1-s} t} \leq 1$, and similarly obtain
    \begin{align*}
    I_2 \leq \sum_{p=0} ^{\infty} \sum_{\lambda_n \in [2^{p}\Lambda_t, \ 2^{p+1}\Lambda_t]} \lambda_n^{-s}
     |\phi_n  (x_0) |^2
    e^{-2\lambda_n^{1-s} t}
    \overset{\eqref{e:dyadic-local-weyl}}{\leq} C_{d,s} \Lambda_t ^{-s} \sum_{p=0} ^{\infty} 2^{-ps}  \left( 2^{p}\Lambda_t\right)^{d/2} = C_{d,s} t^{-\beta},
    \end{align*}
where in the last equality, we used the fact that when $s > d/2$, the series $\sum_{p=0} ^{\infty} 2^{-p(s-d/2)}$ is convergent and equals to a constant $C_{d,s}$ depending only on $d,s$. Combining the estimates for $I_1$ and $I_2$, we obtain the upper bound in \Cref{prop:linearized equation slower t^-1 rate}.

\end{proof}

\section{Wasserstein gradient flows with $o(t^{-1})$-rates: proof of \Cref{thm:wass -1-delta rate intro} and \Cref{cor:uniform perturbation Besov norms}}
\label{s:nonlinear decay close t-1}

In this section, we prove \Cref{thm:wass -1-delta rate intro} by constructing
exact solutions to the Wasserstein gradient flow \eqref{e:wass grad flow}
whose energy has two-sided decay arbitrarily close to $t^{-1}$.
We work on the flat torus $M=\TT^d=\bigl(\bR/(2\pi\bZ)\bigr)^{\otimes d}$ with the inverse fractional Laplacian kernel
$\cK=(-\Delta)^{-s}\circ\Pi_0$ introduced in
\Cref{s:inverse laplacian kernel}. We begin in \Cref{s:preliminary besov norms} by introducing the perturbation equation and the Besov-space framework and then state the main technical result, \Cref{prop:eulerian uniform estimate}. Section~\ref{s:auxiliary lemmas for blockwise estimate} establishes the auxiliary lemmas used in the proof of \Cref{prop:eulerian uniform estimate}. In \Cref{s:proof of blockwise persistence}, we give a proof of \Cref{prop:eulerian uniform estimate}. Finally, \Cref{s:proof of wass -1-delta rate and uniform-in-time perturbation} proves \Cref{thm:wass -1-delta rate intro} and \Cref{cor:uniform perturbation Besov norms}.


\subsection{Preliminaries for \Cref{thm:wass -1-delta rate intro} and \Cref{prop:eulerian uniform estimate}}\label{s:preliminary besov norms}

Let $M$ be a closed manifold satisfying \eqref{a:normalized uniform measure}. For $h\in L^2(M)$, define
\begin{align}\label{e:def Pi D operators}
    \Pi_0 h \coloneqq h-\int_Mh\,\de x, \quad \cD[h] \coloneqq\nabla(-\Delta)^{-1}(\Pi_0h),
\end{align}
and for the present kernel $\cK=(-\Delta)^{-s}\circ\Pi_0$, define
\begin{align}\label{e:def A Q operators}
    \cA[h] \coloneqq-\Delta\cK[h]
    =(-\Delta)^{1-s}(\Pi_0h), \quad 
    \cQ[h] \coloneqq\nabla\cK[h]
    =\nabla(-\Delta)^{-s}(\Pi_0h).         
\end{align}
Then
\begin{align}\label{e:A Q D relations}
    \cQ=\cD\cA,
    \qquad
    \operatorname{div}\cQ=-\cA,
\end{align}
and the energies become
\begin{align}\label{e:inverse Laplacian energy formula}
    \mathsf{E}[u]
    =\frac{1}{2} \int_M u \cdot \cK[u]  \, \de x , \quad \text{ and }\quad 
    \mathsf{E}_{\min}=0.
\end{align}

We write
\begin{align*}
    L_0^2(M)
    \coloneqq
    \left\{h\in L^2(M):\int_Mh\,\de x=0\right\}.
\end{align*}
Let $f_0\in L_0^2(M)\cap L^\infty(M)$, let $\varepsilon>0$, and suppose
that $u_0=1+\varepsilon f_0$ is nonnegative. Writing $u_t=1+\varepsilon f_t$, where $u_t$ solves \eqref{e:wass grad flow}. Then \eqref{e:wass grad flow}, or equivalently \eqref{e:wass grad flow perturbation intro}, becomes
\begin{align}\label{e:wass grad flow perturbation}
    \partial_tf_t + \cA[f_t]
    = \varepsilon \langle \cQ[f_t], \nabla f_t \rangle - \varepsilon f_t \cA[f_t]  =  \varepsilon \operatorname{div}(f_t \cQ[f_t]).
\end{align}
Because \eqref{e:wass grad flow} preserves mass, $f_t$ remains
mean-zero. Moreover, the energies satisfy
\begin{align}\label{e:energy ut ft}
    \mathsf{E}[u_t]
    =\frac{\varepsilon^2}{2}
    \langle f_t,(-\Delta)^{-s}f_t\rangle
    = \varepsilon^2 \mathsf{E}[f_t],
\end{align}
where $\mathsf{E}[f_t]$ is equivalent to the homogeneous Sobolev norm $\|f_t\|_{\dot H^{-s}}^2$ defined in \eqref{e:sobolev norm}. That is, there is a constant $C_s > 1$, such that
    \begin{align}\label{e:equivalence sobolev norm energy}
        C_s ^{-1} \cdot \|f_t\|_{\dot H^{-s}}^2 \leq \mathsf{E}[f_t] \leq C_s \cdot \|f_t\|_{\dot H^{-s}}^2 .
    \end{align}
Also, recall from \eqref{e:linearized flow} and \eqref{e:linearized semigroup} that its linearization at $\sigma = 1 \de x$ solves
\begin{align}\label{e:linearized flow in nonlinear section}
    \partial_t\overline f_t+\cA[\overline f_t]=0,
    \qquad
    \overline f_0=f_0,
\end{align}
and hence $\overline f_t=e^{-t\cA}f_0$.
The derivative $\nabla f_t$ in the nonlinear transport term $\langle \cQ[f_t], \nabla f_t \rangle$ in \eqref{e:wass grad flow perturbation} prevents a direct estimate
of \eqref{e:wass grad flow perturbation} in the same Besov space \eqref{e:Besov norm def}, and results in the loss of one derivative in the uniform comparison between $f_t$ and $\overline f_t$ in \Cref{cor:uniform perturbation Besov norms}. In order to prove the persistence of energy across eigenfunction blocks as discussed after \Cref{thm:wass -1-delta rate intro}, we then seek to analyze the evolution equation of each eigenfunction block of $f_t$. This leads to the following Littlewood--Paley decomposition used throughout
the section. 

We define a smooth Littlewood--Paley decomposition in the following way. Choose a nonincreasing function
$\chi\in C_c^\infty([0,\infty))$ such that
\begin{align*}
    0\leq\chi\leq1,
    \quad
    \chi(\lambda)=1 \text{ for }0\leq\lambda\leq\frac32,
    \quad
    \chi(\lambda)=0 \text{ for }\lambda\geq2, \quad -4 \leq \chi' \leq 0,
\end{align*}
and set
\begin{align}\label{e:eta cut-off}
    \eta(\lambda)\coloneqq\chi(\lambda/2)-\chi(\lambda).
\end{align}
Then $\eta\in C_c^\infty((1,4))$, $\eta=1$ on $[2,3]$, and
\begin{align}\label{e:partition of unity}
    \chi(\lambda)+\sum_{k=0}^\infty\eta(2^{-k}\lambda)=1,
    \quad \forall \lambda\geq0.
\end{align}
Let $\{\phi_n\}_{n\geq0}$ be the orthonormal eigenbasis of $-\Delta$ with eigenvalues $\{\lambda_n\}_{n\geq0}$ in \Cref{a:kernel}. For any $h \in L^2(M)$, we write its spectral expansion as
\begin{align}\label{e:L^2 spectral expansion}
        h = \sum_{n=0} ^\infty b_n \phi_n, \quad \text{ with } b_n \coloneqq \int_M h \phi_n \, \de x.
    \end{align}
We then define
\begin{align}\label{e:partition of unity projection}
        \proj_{-1} h \coloneqq \sum_{n=0} ^\infty \chi\left(\sqrt{\lambda_n} \right) b_n \phi_n, \quad \text{ and } \quad 
        \proj_k h \coloneqq \sum_{n=0} ^{\infty} \eta\left( 2^{-k} \sqrt{\lambda_n} \right)b_n \phi_n, \quad   \forall k \in \bN.
    \end{align}
Thus, combining \eqref{e:partition of unity} and \eqref{e:partition of unity projection}, we have that
\begin{align}\label{e:partition of L^2 function}
    h=\sum_{k=-1}^\infty\proj_kh.
\end{align}
We use the notations
\begin{align}\label{e:dyadic block and partial sum notation}
    h_k\coloneqq\proj_kh,
    \quad
    \bfS_k h\coloneqq\sum_{j= -1} ^k h_j, \text{ for } k \geq -1,
    \quad
    \bfS_k h\coloneqq0\quad\text{if }k<-1.
\end{align}
The smoothness of the cutoff is used in the multiplier-difference estimate
in \Cref{lem:torus dyadic Fourier commutator}; the other dyadic
interactions are controlled by their exact Fourier supports.

Since each $\proj_k$ has finite spectral support, it extends
canonically to the space of distributions $\mathcal D'(M)$. For $\rho\in\bR$ and
$h\in\mathcal D'(M)$, define the Besov norm
\begin{align}\label{e:Besov norm def}
    \|h\|_{B^\rho}
    \coloneqq
    \sup_{k\geq-1}2^{k\rho}\|\proj_kh\|_2,
\end{align}
and denote by $B^\rho(M)$ the space of distributions with finite
$\|\cdot\|_{B^{\rho}}$-norm. Then, for a time-dependent function $H(t,x)$ on $[0,T]\times M$, we denote $H_k(t) \coloneqq \proj_k H(t,\cdot )$ and define the Chemin--Lerner norm
\begin{align}\label{e:Chemin Lerner norm def}
    \|H\|_{B_T^\rho}
    \coloneqq
    \sup_{k\geq-1}2^{k \rho}
    \int_0^T\|H_k(t)\|_2\,\de t.
\end{align}
We also set
\begin{align}\label{e:uniform in time Besov norm}
    \|H\|_{\BB_T^\rho}
    \coloneqq
    \sup_{0\leq t\leq T}\|H(t,\cdot)\|_{B^\rho}
    +\|\cA [H]\|_{B_T^\rho}.
\end{align}
Finally, for any $h \in L_0^2(M)$, we use the following equivalent definition of the standard homogeneous Sobolev norms $\dot H^{\rho}$:
\begin{align}\label{e:sobolev norm}
    \|h\|_{\dot H^\rho}^2
    \coloneqq
    \sum_{k=-1}^\infty2^{2\rho k}\|h_k\|_2^2.
\end{align}
We write
\begin{align}\label{e:def r inverse Laplacian}
    r\coloneqq2(s-1).
\end{align}
All technical lemmas below for the proof of \Cref{thm:wass -1-delta rate intro} require only
    \begin{align*}
        s>1, \quad \text{ and } \sigma> d/2.
    \end{align*}

In particular, when $M = \TT^d = \left(\bR / (2\pi \bZ) \right)^{\otimes d}$, a standard orthonormal basis of Laplacian eigenfunctions is $\{e^{i\langle \xi, \, x \rangle}\}_{\xi \in \bZ^d}$ with eigenvalues $\{|\xi|^2\}_{\xi \in \bZ^d}$. For any $h \in L^2(\TT^d)$, we can write its spectral expansion/Fourier expansion \eqref{e:L^2 spectral expansion} into the form
    \begin{align}\label{e:h torus Fourier expansion}
        h(x) = \sum_{\xi \in \bZ^d}\widehat{h}(\xi)e^{i \langle \xi, \, x \rangle}.
    \end{align}
For $k\geq-1$, define
\begin{align}\label{e:littlewood cutoff indices}
    \vartheta_k(r)
    \coloneqq
    \begin{cases}
        \chi(r), & k=-1,\\
        \eta(2^{-k}r), & k\geq0.
    \end{cases}
\end{align}
Thus, by \eqref{e:L^2 spectral expansion},
    \begin{align}\label{e:hk expansion}
             h_{k}(x)= \proj_k h (x) = \sum_{\xi \in \bZ^d}\widehat{h_{k}}(\xi)e^{i \langle \xi, \, x \rangle} = \sum_{\xi \in \bZ^d} \vartheta_k(|\xi|)\widehat{h}(\xi) e^{i \langle \xi, \, x \rangle} .
        \end{align}
We thus have
    \begin{align}\label{e:Fourier support of littlewood partition}
        \operatorname{supp} \widehat{h_{k}} \subseteq \{\xi \in \bZ^d \mid 2^{k} < |\xi| < 2^{k+2} \} \text{ for } k \geq 0, \quad \text{ and } \operatorname{supp} \widehat{h_{-1}} \subseteq \{\xi \in \bZ^d \mid 0 \leq |\xi| < 2 \} .
    \end{align}
Similarly, for a vector field $V \in L^2 (M)$, we also write
    \begin{align}\label{e:Vk expansion}
             V_{k}(x)= \proj_k V (x) = \sum_{\xi \in \bZ^d} \widehat{V_{k}}(\xi)e^{i \langle\xi, \, x \rangle} =  \sum_{\xi \in \bZ^d} \vartheta_k(|\xi|)\widehat{V}(\xi) e^{i \langle \xi, \, x \rangle} .
        \end{align}
The notation $V(h)$ denotes the function $(V(h))(x) \coloneqq \langle V(x) ,\nabla h (x)\rangle$.

We now state \Cref{prop:eulerian uniform estimate} as the main technical ingredient in the proof of \Cref{thm:wass -1-delta rate intro}.

\begin{proposition}[Uniform bounded Besov norms and persistence of dyadic blocks]
\label{prop:eulerian uniform estimate}
Let $M=\TT^d$. Let $s>1$, $\sigma>d/2$, and denote $r \coloneqq 2(s-1)$. For any $k \geq -1$, denote 
    \begin{align*}
        \omega_k\coloneqq2^{-rk} .
    \end{align*}
Let $0<\varepsilon\leq1$. There are constants $c_0,C_0>0$,
depending only on $s,\sigma,d$, such that
the following holds. For every $f_0\in L_0^2(M)\cap B^\sigma(M)$ satisfying
\begin{align}\label{e:eulerian smallness condition}
    \varepsilon\|f_0\|_{B^\sigma}\leq c_0,
\end{align}
equation \eqref{e:wass grad flow perturbation} admits a global
solution $f_t$ on $[0,\infty) \times M$. Indeed, for any $T\geq0$, define
\begin{align}\label{e:def Eulerian X Y}
    \cX_T\coloneqq\sup_{0\leq t\leq T}\|f_t\|_{B^\sigma},
    \qquad
    \cY_T\coloneqq\|\cA [f_t]\|_{B_T^\sigma}.
\end{align}
Then,
\begin{align}\label{e:eulerian global bound}
    \sup_{T \geq 0} (\cX_T+\cY_T)\leq C_0\cX_0.
\end{align}
Moreover, there are constants $c_1,C_1,\tau_0>0$,
depending only on $s,\sigma,d$, such that for any
$k\geq-1$ and any $t \geq 0$,
\begin{align}\label{e:dyadic persistence quantitative}
    \|\proj_k f_t \|_2
    \geq
    e^{-C_1(\omega_k t+\varepsilon \cX_0)}
    \|\proj_k f_0\|_2
    -
    C_1  \varepsilon \cX_0^2 \cdot 2^{-\sigma k}.
\end{align}
In particular, if, for some constant $\gamma>0$, the initial datum satisfies
\begin{align}\label{e:dyadic persistence initial}
    \|\proj_k f_0\|_2\geq\gamma2^{-\sigma k},
    \qquad
    \varepsilon \cX_0^2\leq c_1\gamma,
\end{align}
then
\begin{align}\label{e:dyadic persistence simple}
    \|\proj_k f_t\|_2
    \geq
    \frac{1}{2}\gamma2^{-\sigma k},
    \quad
    \text{ for }0\leq t\leq\tau_0\omega_k^{-1}.
\end{align}
\end{proposition}

\subsection{Auxiliary lemmas and proofs for \Cref{prop:eulerian uniform estimate}}\label{s:auxiliary lemmas for blockwise estimate}

The proof of \Cref{prop:eulerian uniform estimate} is a combination of the following classical harmonic analysis results. Some of these estimates are standard and we put them in \Cref{s:harmonic analysis lemmas}. The following lemmas focus on the estimates involving Besov norms.

\begin{lemma}[Product Lemma]\label{lem:Besov product lemma}
    Take any  $\sigma > d/2$. There is a constant $C=C(\sigma, d, M) >0$, such that for any $h,g \in B^{\sigma}(M)$,    
        \begin{align}\label{e:Besov product lemma}
            \|hg\|_{B^\sigma} \leq C
\|h\|_{B^\sigma}\|g\|_{B^\sigma}.
        \end{align}
    Furthermore, for any time-dependent functions $H(t,x),G(t,x)$ on $ [0,T] \times M $, 
        \begin{align}\label{e:Besov product lemma in time}
            \|HG\|_{B_T ^\sigma} \leq C
\left(\sup_{0\le t\le T}\|H(t,\cdot)\|_{B^\sigma}\right) \|G\|_{B_T ^\sigma}.
        \end{align}
\end{lemma}
\begin{proof}
    We only prove \eqref{e:Besov product lemma in time} as the proof of \eqref{e:Besov product lemma} is similar. First, by the Bernstein-type estimates \Cref{lem:Bernstein inequalities}, we have that $B^{\sigma}(M) \hookrightarrow L^{\infty}(M)$ when $ \sigma > d/2$. Indeed, for any $ h \in B^{\sigma}(M)$,
        \begin{align}\label{e:Besov norm control infinity norm}
\|h\|_{\infty} \overset{\eqref{e:partition of L^2 function}}{\leq}
\sum_{k=  -1} ^{\infty} \|\proj_kh\|_\infty \overset{\Cref{lem:Bernstein inequalities}}{\leq}
C \sum_{k=  -1} ^{\infty} 2^{kd/2} \|\proj_k h\|_2 \overset{\eqref{e:Besov norm def}}{\leq}
C \|h\|_{B^\sigma} \sum_{k=  -1} ^{\infty} 2^{k(d/2-\sigma)} = C_{M,\sigma} \|h\|_{B^\sigma},
        \end{align}
where $C_{M,\sigma}>0$ is a constant depending only on $\sigma,d,M$ as $\sum_{k=  -1} ^{\infty} 2^{k(d/2-\sigma)} < \infty$ when $\sigma > d/2$. Similar to the proof of \eqref{e:Besov norm control infinity norm}, for any $k \geq -1$ and the partial sums $\bfS_k$ in \eqref{e:dyadic block and partial sum notation}, we have that
    \begin{align}\label{e:Besov norm control infinity norm partial sum}
        \| \bfS_k h \|_{\infty} \leq C_{M,\sigma} \|h\|_{B^{\sigma}},
    \end{align}
and 
    \begin{align}\label{e:Besov norm control infinity norm partial sum in time}
        \int_0 ^T \| \bfS_k G(t,\cdot) \|_{\infty} \, \de t \leq C_{M,\sigma} \|G\|_{B_T ^\sigma} .
    \end{align}

Next, for any time-dependent functions $H,G$ on  $[0,T] \times M$, we have the following Bony decomposition for their product $HG$:
    \begin{align}\label{e:Bony decomposition}
        HG= \Gamma_H G+ \Gamma_G H+R(H,G),
    \end{align}
where
    \begin{align*}
        \Gamma_H G \coloneqq \sum_{\ell \geq -1} \bfS_{\ell -4}H \cdot \proj_\ell G, \quad \Gamma_G H \coloneqq \sum_{\ell \geq -1} \bfS_{\ell -4}G \cdot \proj_\ell H, \quad R(H,G) \coloneqq \sum_{\substack{j,\ell \geq -1 \\ |j-\ell| \leq 3}}  \proj_j H \cdot \proj_\ell G.
    \end{align*}
We first estimate $\|\Gamma_H G\|_{B_T ^\sigma}$. For any $k,\ell \geq -1 $ and $t \in [0,T]$, by \Cref{lem:almost orthogonal littlewood paley}, for any $\xi \in \bZ_+$,
    \begin{align*}
        \begin{split}
            \|\proj_k \left( \bfS_{\ell -4} H(t,\cdot) \cdot \proj_\ell G(t,\cdot) \right)\|_2 &\leq C_{\xi,M} 2^{-\xi(k-\ell)_+} \| \bfS_{\ell -4} H(t,\cdot) \|_{\infty} \|\proj_\ell G(t,\cdot) \|_2 
            \\  &\overset{\eqref{e:Besov norm control infinity norm partial sum}}{\leq}  C_{\sigma,\xi,M} 2^{-\xi(k-\ell)_+} \|  H(t,\cdot) \|_{B^{\sigma}} \|\proj_\ell G(t,\cdot) \|_2.
        \end{split}
    \end{align*}
Thus,
    \begin{align*}
        \begin{split}
                    &2^{k \sigma} \int_0^T
\left\|\proj_k \left[ \Gamma_H G\right] \right\|_2\,\de t
\leq C_{\sigma,\xi,M} \cdot 
\sup_{0\leq t\leq T}\|H(t,\cdot)\|_{B^\sigma}
\sum_{\ell \geq -1}
2^{k\sigma}
2^{-\xi(k-\ell)_+}
\int_0^T\|\proj_\ell G(t,\cdot )\|_2\, \de t
\\  &\leq C_{\sigma,\xi,M} \cdot 
\sup_{0\leq t\leq T}\|H(t,\cdot)\|_{B^\sigma} \cdot \|G\|_{B_T ^\sigma} \cdot 
\sum_{\ell \geq -1}
2^{(k-\ell)\sigma}
2^{-\xi(k-\ell)_+}.
        \end{split}
    \end{align*}
Now, set $\xi = \lfloor 2\sigma \rfloor +2 >  2\sigma+1$. We have that
    \begin{align}\label{e:low high sum}
        \begin{split}
            &\sum_{\ell \geq -1}
2^{(k-\ell)\sigma}
2^{-\xi(k-\ell)_+}  =  \sum_{\ell \in [ -1,k-1]}
2^{(k-\ell)\sigma}
2^{-\xi(k-\ell)_+} + \sum_{\ell \geq k}
2^{(k-\ell)\sigma}
2^{-\xi(k-\ell)_+}
    \\  &= \sum_{q = 1} ^{k+1} 2^{-(\xi-\sigma)q} + \sum_{q=0} ^{\infty} 2^{-q \sigma} \leq 2 \sum_{q=0} ^{\infty} 2^{-q \sigma} = \frac{2}{1-2^{-\sigma}} < \infty,
        \end{split}
    \end{align}
which is a constant $C_{\sigma}$ independent of $k$. Thus, we obtain that
    \begin{align*}
        \|\Gamma_H G\|_{B_T ^\sigma} \leq C_{\sigma,M} \cdot 
\sup_{0\leq t\leq T}\|H(t,\cdot)\|_{B^\sigma} \cdot \|G\|_{B_T ^\sigma} .
    \end{align*}
Next, we estimate $\|\Gamma_G H\|_{B_T ^\sigma}$. Similarly, for any $k,\ell \geq -1 $ and $t \in [0,T]$, by \Cref{lem:almost orthogonal littlewood paley}, for any $\xi \in \bZ_+$,
    \begin{align*}
        \begin{split}
            \|\proj_k \left( \bfS_{\ell -4} G(t,\cdot) \cdot \proj_\ell H(t,\cdot) \right)\|_2 &\leq C_{\xi,M} 2^{-\xi(k-\ell)_+} \| \bfS_{\ell -4} G(t,\cdot) \|_{\infty} \|\proj_\ell H(t,\cdot) \|_2 
            \\  &\overset{\eqref{e:Besov norm def}}{\leq}  C_{\xi,M} 2^{-\xi(k-\ell)_+} 2^{-\sigma \ell} \|H(t,\cdot)\|_{B^\sigma} \| \bfS_{\ell -4} G(t,\cdot) \|_{\infty} .
        \end{split}
    \end{align*}
Thus, set $\xi = \lfloor 2\sigma \rfloor +2 >  2\sigma+1$ again, 
    \begin{align*}
        \begin{split}
                    &2^{k \sigma} \int_0^T
\left\|\proj_k \left[ \Gamma_G H\right] \right\|_2\,\de t
\leq C_{\sigma,M} \cdot 
\sup_{0\leq t\leq T}\|H(t,\cdot)\|_{B^\sigma}
\sum_{\ell \geq -1}
2^{(k-\ell)\sigma}
2^{-\xi(k-\ell)_+}
\int_0^T \| \bfS_{\ell -4} G(t,\cdot) \|_{\infty}\, \de t
\\  &\overset{\eqref{e:Besov norm control infinity norm partial sum in time}}{\leq} C_{\sigma,M} \cdot 
\sup_{0\leq t\leq T}\|H(t,\cdot)\|_{B^\sigma} \cdot \|G\|_{B_T ^\sigma} \cdot 
\sum_{\ell \geq -1}
2^{(k-\ell)\sigma}
2^{-\xi(k-\ell)_+}
\\  &\overset{\eqref{e:low high sum}}{=} C_{\sigma, M} \cdot 
\sup_{0\leq t\leq T}\|H(t,\cdot)\|_{B^\sigma} \cdot \|G\|_{B_T ^\sigma} .
        \end{split}
    \end{align*}
Thus, we also obtain that
    \begin{align*}
        \|\Gamma_G H\|_{B_T ^\sigma} \leq C_{\sigma,M} \cdot 
\sup_{0\leq t\leq T}\|H(t,\cdot)\|_{B^\sigma} \cdot \|G\|_{B_T ^\sigma} .
    \end{align*}
We finally estimate $\|R(H,G)\|_{B_T ^\sigma}$. For any $j,\ell \geq -1$ with $|j-\ell| \leq 3$ and any $t \in [0,T]$,
    \begin{align*}
            \|\proj_j H(t,\cdot)\|_{\infty} \overset{\Cref{lem:Bernstein inequalities}}{\leq} C_M 2^{jd/2} \|\proj_j H(t,\cdot)\|_{2}\overset{\eqref{e:Besov norm def}}{\leq} C_M 2^{-j(\sigma-d/2)} \|H(t,\cdot)\|_{B^\sigma} \leq C_{\sigma,M} 2^{-\ell (\sigma-d/2)} \|H(t,\cdot)\|_{B^\sigma}.
    \end{align*}
Therefore, for any $k \geq -1$, by \Cref{lem:almost orthogonal littlewood paley}, because $|j-\ell| \leq  3$,
    \begin{align*}
        \begin{split}
            & 2^{k \sigma} \int_0^T \|\proj_k \left(\proj_j H \cdot \proj_\ell G \right)\|_2 \, \de t \leq C_{\xi,M} 2^{k \sigma}  2^{-\xi (k-\ell)_+} \int_0^T \|\proj_j H(t,\cdot)\|_{\infty} \| \proj_\ell G(t,\cdot)\|_2 \, \de t
            \\ &\leq C_{\sigma,\xi,M} 2^{-\xi (k-\ell)_+ } \cdot 2^{k\sigma-\ell (\sigma-d/2)} \sup_{0\leq t\leq T}\|H(t,\cdot)\|_{B^\sigma} \int_0 ^T \| \proj_\ell G(t,\cdot)\|_2 \, \de t
            \\ &\leq C_{\sigma,\xi,M} 2^{-\xi (k-\ell)_+ } \cdot 2^{(k-\ell)\sigma-\ell (\sigma-d/2)} \sup_{0\leq t\leq T}\|H(t,\cdot)\|_{B^\sigma} \cdot \|G\|_{B_T ^\sigma}
            \\ &\leq C_{\sigma,\xi,M} 2^{-\xi (k-\ell)_+ } \cdot 2^{(k-\ell)\sigma} \sup_{0\leq t\leq T}\|H(t,\cdot)\|_{B^\sigma} \cdot \|G\|_{B_T ^\sigma}.
        \end{split}
    \end{align*}
where we used the trivial bound $\sigma -d/2 >0$ in the last inequality.
Thus, set $\xi = \lfloor 2\sigma \rfloor +2 >  2\sigma+1$ again,
    \begin{align*}
        \begin{split}
            &2^{k \sigma} \int_0^T \|\proj_k \left[R(H,G) \right]\|_2 \, \de t \leq  C_{\sigma,M} \cdot 
\sup_{0\leq t\leq T}\|H(t,\cdot)\|_{B^\sigma} \cdot \|G\|_{B_T ^\sigma} \cdot 
\sum_{\ell \geq -1}
2^{(k-\ell)\sigma}
2^{-\xi(k-\ell)_+}
\\  &\overset{\eqref{e:low high sum}}{=} C_{\sigma, M} \cdot 
\sup_{0\leq t\leq T}\|H(t,\cdot)\|_{B^\sigma} \cdot \|G\|_{B_T ^\sigma} .
        \end{split}
    \end{align*}
Thus, we finally obtain that
    \begin{align*}
        \|R(H,G)\|_{B_T ^\sigma} \leq C_{\sigma,M} \cdot 
\sup_{0\leq t\leq T}\|H(t,\cdot)\|_{B^\sigma} \cdot \|G\|_{B_T ^\sigma} .
    \end{align*}
This finishes the proof of \eqref{e:Besov product lemma in time}, and the proof of \eqref{e:Besov product lemma} follows similarly.
\end{proof}

\begin{lemma}[Dyadic Fourier commutator]\label{lem:torus dyadic Fourier commutator}
Let $M = \TT^d$. Let $h \in L^2(M)$ be a function and $V \in L^2(M)$ be a vector field. Adopt the notations in \eqref{e:hk expansion} and \eqref{e:Vk expansion}. Then, for any $k,j,\ell\geq-1$,
\begin{align}\label{e:torus elementary commutator}
    \left\|[\proj_k,V_j]h_\ell \right\|_2
    \leq
    C2^{\ell-k}2^{j(d/2+1)}
    \|V_j\|_2\|h_\ell\|_2,
\end{align}
where $[\proj_k,V_j]h_\ell
    \coloneqq
    \proj_k(V_j( h_\ell))
    -V_j\left(\proj_k h_\ell \right)$, and $C$ is a constant depending only on $d$.
\end{lemma}
\begin{proof}
By the mean value theorem for $\vartheta_k$, there is a universal constant $C$ only depending on $\chi$, such that uniformly for $k\geq-1$ and $r,r'\geq0$,
\begin{align}\label{e:uniform dyadic multiplier Lipschitz}
    |\vartheta_k(r)-\vartheta_k(r')|
    \leq
    C \cdot 2^{-k}|r-r'|.
\end{align}
A direct computation gives when $\omega \in \bZ^d$,
\begin{align*}
    \widehat{[\proj_k,V_j]h_\ell}(\omega) =
    \sum_{\zeta+\xi=\omega}
        i\langle\widehat{V_j}(\zeta),\xi\rangle
        \bigl(
            \vartheta_k(|\zeta+\xi|)
            -
            \vartheta_k(|\xi|)
        \bigr)
        \widehat{h_\ell}(\xi).
\end{align*}
By \eqref{e:uniform dyadic multiplier Lipschitz}, $ \left|
        \vartheta_k(|\zeta+\xi|)-\vartheta_k(|\xi|)
    \right|
    \leq
    C2^{-k}\bigl||\zeta+\xi|-|\xi|\bigr|
    \leq
    C2^{-k}|\zeta|$.
Combining $|\xi| \leq 2^{\ell +2}$ by \eqref{e:Fourier support of littlewood partition}, there is a universal constant $C$ such that
\begin{align*}
    \left|\widehat{[\proj_k,V_j]h_\ell}(\omega)\right|
    \leq
    C2^{\ell-k}
    \sum_{\zeta+\xi=\omega}
        |\zeta|\,|\widehat{V_j}(\zeta)|\,|\widehat{h_\ell}(\xi)|.
\end{align*}
Combining Parseval's identity and Young's convolution inequality,
\begin{align*}
    \left\|[\proj_k,V_j]h_\ell \right\|_2
    \leq
    C2^{\ell-k}
    \left(
         \sum_{\zeta \in \bZ^d}
        |\zeta|\,|\widehat{V_j}(\zeta)|
    \right)
    \|h_\ell\|_2.
\end{align*}
For every $j\geq-1$, apply \eqref{e:Fourier support of littlewood partition} to $\widehat{V_j}$, and use the Cauchy--Schwarz inequality and Parseval's identity, 
\begin{align*}
    \sum_{\zeta \in \bZ^d}
        |\zeta|\,|\widehat{V_j}(\zeta)| \leq
    \left(
        \sum_{|\zeta |\leq 4 \cdot 2^j}|\zeta|^2
    \right)^{1/2}
    \|V_j\|_2\leq
    C\cdot 2^{j(d/2+1)}\|V_j\|_2,
\end{align*}
where $C $ is a constant depending on $d$. This completes the proof of \Cref{lem:torus dyadic Fourier commutator}.
\end{proof}

We now prove the dyadic transport estimate \Cref{lem:torus dyadic transport}. This is the point at which
working on the torus makes the argument transparent. For any function $h$ and vector field $V$ on $M = \TT^d$, we denote
    \begin{align}\label{e:def dyadic transport}
        \cT_k[V,h]
    \coloneqq [\proj_k,V](h) = 
    \proj_k\bigl(V(h)\bigr) - V(h_k).
    \end{align}

\begin{lemma}[Dyadic transport remainder on $\TT^d$]
\label{lem:torus dyadic transport}
Let $M = \TT^d$. Take any $T \geq 0$, $\sigma> d/2$. Then for a time-dependent function $H(t,x)$ and a vector field $V(t,x)$ on $[0,T]\times M$, 
\begin{align}\label{e:torus dyadic transport estimate}
    \sup_{k\geq-1}2^{\sigma k}
    \int_0^T\|\mathcal T_k[V,H](t,\cdot)\|_2\,\de t                 \leq
    C
    \left(\sup_{0\leq t\leq T}\|H(t,\cdot)\|_{B^\sigma}\right)
    \|V\|_{B_T^{\sigma+1}},
\end{align}
where $C$ is a constant depending only on $\sigma,d$.
\end{lemma}
\begin{proof}
Denote $\cH = \sup_{0\leq t\leq T}\|H(t,\cdot)\|_{B^\sigma}$, $
    \cV =
    \|V\|_{B_T^{\sigma+1}}$. 
Fix a time $t \in [0,T]$ and a $k \geq -1$. Using \eqref{e:partition of L^2 function}, we can expand $H$ and $V$ into the forms
\begin{align*}
    H=\sum_{\ell\geq-1}H_\ell,
    \qquad
    V=\sum_{j\geq-1}V_j .
\end{align*}
Then,
\begin{align}\label{e:full transport direct commutator}
    \cT_k[V,H]
    =
    \sum_{j,\ell\geq-1}
    [\proj_k,V_j]H_\ell = \sum_{j\leq\ell-L}
    [\proj_k,V_j]H_\ell + \sum_{j>\ell-L}
    [\proj_k,V_j]H_\ell \eqqcolon I_1 + I_2,
\end{align}
where we fix $L=6$. In the following, all summations are for indices $j,\ell \geq -1$.

We first estimate $I_1$. In this region, $V_j$ has
frequency much lower than $H_\ell$.  Write $\zeta$ for a frequency of
$V_j$ and $\xi$ for a frequency of $H_\ell$. By
\eqref{e:Fourier support of littlewood partition}, for $\ell\geq L+j > 0$,
\begin{align*}
    |\zeta|<2^{\ell-L+2},
    \quad
    2^{\ell}<|\xi|<2^{\ell+2}.
\end{align*}
Hence every frequency $\omega=\zeta+\xi$ of
$V_j\left( H_{\ell} \right)$ satisfies
\begin{align}\label{e:I1 product Fourier annulus}
    \left(1-2^{2-L}\right)2^{\ell}
    <|\omega|<
    \left(4+2^{2-L}\right)2^{\ell}.
\end{align}
Because $L=6$, the annulus in
\eqref{e:I1 product Fourier annulus} is disjoint from the Fourier
support of $\proj_k$ if $|\ell-k|\geq3$. Thus,
\begin{align*}
    \proj_k\big(V_j\left( H_\ell \right)\big)=0,
    \quad\text{if }|\ell-k|\geq3.
\end{align*}
Also,
    \begin{align*}
        \proj_k H_\ell=0,
    \quad\text{if }|\ell-k|\geq2.
    \end{align*}
Therefore, the summation in $I_1$ is for $\ell$ such that $|\ell - k | \leq 2$, and thus
    \begin{align}\label{e:control I1 in dyadic transport decomposition}
        \begin{split}
            &2^{\sigma k} \int_0^T \| I_1 \|_2 \, \de t \leq   2^{\sigma k} \sum_{|\ell - k| \leq 2}\sum_{j\leq\ell-L}
    \int_0^T \| [\proj_k,V_j]H_\ell \|_2 \, \de t
    \\
    &\overset{ \Cref{lem:torus dyadic Fourier commutator}}{\leq}
    C_d \cdot 2^{\sigma k}
    \sum_{|\ell-k|\leq2}2^{\ell-k}
       \cdot \sup_{0\leq t\leq T}\|H_\ell(t)\|_2 \cdot
    \sum_{j\leq k+2-L}2^{j(d/2+1)}
        \int_0^T\|V_j(t)\|_2\,\de t          \\
    &\leq
    C_d  
    \sum_{|\ell-k|\leq2}2^{\sigma(k-\ell)+\ell-k} \cdot \cH \cdot 
    \sum_{j\leq k+2-L}2^{j(d/2-\sigma)} \cdot \cV    \leq C_{\sigma,d} \cdot \cH \cdot \cV,
        \end{split}
    \end{align}
where the last geometric series is bounded by a constant $C_{\sigma,d}$ because
$\sigma>d/2$.

We then estimate $I_2$. First,
\begin{align*}
        I_2 = \sum_{j>\ell-L}
        [\proj_k,V_j]H_\ell
    = \sum_{j>\ell-L}
        -V_j(\proj_kH_\ell)
    +
        \sum_{j>\ell-L}
        \proj_k\bigl(V_j(H_\ell)\bigr)
    \eqqcolon I_{3}+I_{4}.
\end{align*}
For $I_3$, since $\proj_kH_\ell=0$ if $|\ell-k|\geq 2$, the condition
$j>\ell-L$ implies $j\geq k-L$ in the sum defining $I_3$. Thus,
\begin{align}\label{e:control I3 in dyadic transport decomposition}
    \begin{split}
        &2^{\sigma k} \int_0^T \| I_3 \|_2 \, \de t \leq  2^{\sigma k} \sum_{|\ell - k| \leq 2}\sum_{j\geq k-L}
    \int_0^T \|V_j(\proj_kH_\ell) \|_2 \, \de t
    \\
    &\leq
    2^{\sigma k}
    \sum_{|\ell-k|\leq2}
        \sup_{0\leq t\leq T}
        \|\nabla\proj_k H_\ell(t)\|_\infty
    \sum_{j\geq k-L}
        \int_0^T\|V_j(t)\|_2\,\de t
        \\ & \overset{\Cref{lem:Bernstein inequalities}}{\leq} C_{d} \cdot 2^{\sigma k}
    \sum_{|\ell-k|\leq2}
        \sup_{0\leq t\leq T}
        2^{(1+d/2)k} \cdot \|\proj_k H_\ell(t)\|_2
    \sum_{j\geq k-L}
        \int_0^T\|V_j(t)\|_2\,\de t
    \\
    &\leq
    C_d \cdot \cH \cdot 
    2^{(1+d/2)k}
    \sum_{j\geq k-L}
        2^{-(\sigma+1)j} \cdot \cV
    \leq
    C_{\sigma,d}\cdot \cH \cdot \cV \cdot 
    2^{(d/2-\sigma)k}
    \leq C_{\sigma,d}\cdot \cH \cdot \cV ,
    \end{split}
\end{align}
where the last inequality is because $\sigma > d/2$. For $I_4$,  because $j>\ell-L$, every
frequency $\omega$ of $V_j \left(H_\ell \right)$ satisfies
\begin{align*}
    |\omega| \leq 2^{j+2}+2^{\ell+2} \leq 2^{j+2}+2^{j+L+1}
    <2^{j+L+2}.
\end{align*}
It follows that
\begin{align}\label{e:I4 exact upper restriction}
    \proj_k(V_j \left(H_\ell \right))=0
    \quad\text{if }k\geq j+L+2.
\end{align}
There is no general lower restriction on $k$, since two frequencies from $V_j,H_{\ell}$ may cancel. Also,
\begin{align*}
    \|\proj_k(V_j \left(H_\ell \right))\|_2
    \leq \|V_j(H_\ell)\|_2 \leq
    \|V_j\|_2\|\nabla H_\ell\|_\infty \overset{\Cref{lem:Bernstein inequalities}}{\leq}
    C_d \cdot 2^{(1+d/2)\ell}
    \|V_j\|_2\|H_\ell\|_2.
\end{align*}
Combining this with \eqref{e:I4 exact upper restriction},
\begin{align}\label{e:control I4 in dyadic transport decomposition}
   \begin{split}
       & 2^{\sigma k}\int_0^T \|I_4\|_2 \,\de t \leq C_d \cdot 2^{\sigma k} \sum_{j > k-L - 2} \sum_{\ell < j+ L }  2^{(1+d/2)\ell}
    \sup_{0 \leq t \leq T}\|H_\ell(t)\|_2 \cdot \int_0 ^T  \|V_j(t)\|_2 \, \de t
    \\ &\leq
    C_d \cdot \cH \cdot \cV \cdot 
    2^{\sigma k}
    \sum_{j > k-L - 2} 2^{-(\sigma+1)j} \sum_{\ell < j+ L }
        2^{(1+d/2-\sigma)\ell}
    \\ & \leq C_d \cdot \cH \cdot \cV \cdot 
    2^{\sigma k}
    \sum_{j > k-L - 2} 2^{-\sigma j} \sum_{\ell < j+ L }
        2^{(d/2-\sigma)\ell} 
        \\ &\leq  C_d \cdot \cH \cdot \cV \cdot \sum_{k-j < L+2} 2^{\sigma (k-j)} \cdot \sum_{\ell= -1 } ^{\infty}
        2^{(d/2-\sigma)\ell} \leq C_{\sigma,d} \cdot \cH \cdot \cV
   \end{split}
\end{align}
where in the third inequality, we used $2^{\ell} < 2^{j +L}$ when $\ell < j+L$ and $L=6$, and the last inequality is because by $\sigma > d/2$,
    \begin{align*}
        \sum_{\ell= -1 } ^{\infty}
        2^{(d/2-\sigma)\ell} < \infty, \quad \text{ and } \sum_{k-j < L+2} 2^{\sigma (k-j)} \leq \sum_{m=-\infty} ^{L+1} 2^{\sigma m} < \infty .
    \end{align*}

Combining \Cref{e:control I1 in dyadic transport decomposition,e:control I3 in dyadic transport decomposition,e:control I4 in dyadic transport decomposition} and taking the
supremum over $k\geq-1$ proves
\eqref{e:torus dyadic transport estimate}.
\end{proof}

\subsection{Proof of \Cref{prop:eulerian uniform estimate}}\label{s:proof of blockwise persistence}

We now apply the preceding lemmas to prove \Cref{prop:eulerian uniform estimate}.

\begin{proof}[Proof of \Cref{prop:eulerian uniform estimate}]

Set $v_t\coloneqq\varepsilon\cQ[f_t]$. Then \eqref{e:wass grad flow perturbation} becomes
\begin{align}\label{e:wass grad flow perturbation 2}
\partial_t f_t+\cA[f_t]
=
\operatorname{div}(f_t v_t).
\end{align}
We first justify the regularity needed for the estimates below. Since
$\sigma>d/2$, \eqref{e:Besov norm control infinity norm} implies that
$B^\sigma(M)\hookrightarrow C(M)$. After decreasing $c_0$ if
necessary, \eqref{e:eulerian smallness condition} therefore ensures that
$\varepsilon \|f_0\|_{\infty }\leq1/2$. Because $f_0$ has zero mean,
$u_0=1+\varepsilon f_0$ is a probability density. Hence
\cite[Proposition~2.8]{chizat2026quantitative}, applied to the original
Wasserstein gradient flow \eqref{e:wass grad flow}, gives a solution
$u_t=1+\varepsilon f_t$ on a nontrivial time interval.

Let $T$ lie within the maximal interval of existence time of $u_t$. To avoid assuming
a priori that $\cX_T$ and $\cY_T$ are finite, fix
$\sigma_0\in(d/2,\sigma)$ and take
$\sigma_2\in[\sigma_0,\sigma)$. Choose
$\sigma_1\in(\sigma_2,\min\{\sigma,\sigma_2+1\})$. Since $B^\sigma(M)\hookrightarrow H^{\sigma_1}(M)$, the propagation of Sobolev regularity in
\cite[Proposition~2.14]{chizat2026quantitative} gives $f_t \in L^\infty 
\bigl([0,T],H^{\sigma_1}(M)\bigr)$. The embeddings
$H^{\sigma_1}(M)\hookrightarrow B^{\sigma_2}(M)$ and
$H^{\sigma_1+r}(M)\hookrightarrow B^{\sigma_2}(M)$, together with the
fact that $\cA$ has order $-r$, imply that
\begin{align*}
\cX_T^{\sigma_2}
\coloneqq
\sup_{0\leq t\leq T}\|f_t\|_{B^{\sigma_2}}
<\infty,
\quad
\cY_T^{\sigma_2}
\coloneqq
\|\cA[f_t]\|_{B_T^{\sigma_2}}
<\infty.
\end{align*}
Also, \eqref{e:wass grad flow perturbation 2} implies that $\partial_t f_t \in L^\infty
\bigl([0,T],H^{\sigma_1-1}(M)\bigr)$. The interpolation between $H^{\sigma_1-1}(M)$ and $H^{\sigma_1}(M)$ then gives $f_t \in C
\bigl([0,T],H^{\sigma_2}(M)\bigr)$ for any $\sigma_2 \in (\sigma_1 - 1 , \sigma_1)$, and thus $f_t \in C
\bigl([0,T],B^{\sigma_2}(M)\bigr)$. Thus, $\cX_T^{\sigma_2},\cY_T^{\sigma_2}$ are continuous in $T$. We may therefore apply the rest of the proof below with $\sigma$ replaced by
$\sigma_2$. Because $\sigma_2\in[\sigma_0,\sigma)$, all constants in the
estimates can be chosen depending only on $s,\sigma,d$ and independent of $\sigma_2$. The resulting subcritical analogue of \eqref{e:eulerian global bound 2} is
\begin{align*}
\cX_T^{\sigma_2}+\cY_T^{\sigma_2}
\leq
C_0\cX_0,
\end{align*}
where $C_0$ depends only on $s,\sigma,d$ and is independent of
$T$ and $\sigma_2$. Letting $\sigma_2\to \sigma$ blockwise and then
taking the supremum over the dyadic indices yields the desired $\cX_T+\cY_T\leq C_0\cX_0$ in \eqref{e:eulerian global bound}. For notational simplicity, we write the
estimates below directly with the exponent $\sigma$, with this
subcritical limiting argument understood. We thus assume that $\cX_T$ and $\cY_T$ are finite and continuous in $T$ in the following arguments.

We first analyze the dyadic action of $\cA$. Because $u_t = 1+\varepsilon f_t$ solves the Wasserstein gradient flow \eqref{e:wass grad flow}, $f_t$ remains mean-zero, i.e., $\widehat{f_t}(0) = 0$. The symbol of
\begin{align*}
    \cA=(-\Delta)^{1-s}\circ\Pi_0
\end{align*}
on the nonzero Fourier mode $\xi \in\bZ^d\setminus\{0\}$ is
$|\xi|^{-r}$. For any $k \geq -1$, if $\xi \in\bZ^d\setminus\{0\}$ is a frequency such that $\widehat{\proj_k f_t}(\xi) \neq 0$, then \eqref{e:Fourier support of littlewood partition} implies that
\begin{align}\label{e:uniform dyadic frequency scale}
    2^k < |\xi|< 4\cdot 2^{k}.
\end{align}
Thus, for any $k \geq -1$, Plancherel's identity gives
\begin{align}\label{e:A block coercivity}
    \begin{split}
        C_s ^{-1} \cdot \omega_k \left\|\proj_k f_t \right\|_2 ^2
    &\leq
    \langle\cA [\proj_k f_t],\proj_k f_t\rangle
    \leq
    C_s \cdot \omega_k \left\|\proj_k f_t \right\|_2^2,
    \\
    C_s ^{-1} \cdot \omega_k \left\|\proj_k f_t \right\|_2
    &\leq
    \left\|\cA [\proj_k f_t] \right\|_2
    \leq
    C_s \cdot \omega_k \left\|\proj_k f_t \right\|_2,
    \end{split}
\end{align}
where $C_s > 1$ is a constant depending only on $s$.

Because $v_t=\varepsilon\cQ[f_t] =\varepsilon\cD\cA [f_t]$, $\cD= \nabla(-\Delta)^{-1}\Pi_0$, whose symbol has modulus $|\xi|^{-1}$. Hence, by \eqref{e:uniform dyadic frequency scale}, there is a constant $C_d $ such that for any $j\geq-1$,
\begin{align*}
    \|\proj_j v_t\|_2
    \leq
    C_d \cdot \varepsilon2^{-j}
    \|\proj_j(\cA [f_t])\|_2.
\end{align*}
Thus,
    \begin{align}\label{e:velocity Chemin bound}
        \|v_t\|_{B_T^{\sigma+1}} = \sup_{j\geq-1}2^{j (\sigma+1)}
    \int_0^T \left\|\proj_j v_t \right\|_2\,\de t \leq C_d \cdot \varepsilon \cdot \sup_{j\geq-1}2^{j \sigma} \int_0^T \|\proj_j(\cA [f_t])\|_2 \,\de t = C_d \cdot \varepsilon \cdot \cY_T .
    \end{align}
Since $\operatorname{div}v_t =-\varepsilon\cA[f_t]$, we have that $\operatorname{div} v_t = -\varepsilon \sum_{j=-1} ^{\infty} \proj_j \left(\cA [f_t] \right) $. Thus
\begin{align}\label{e:principal divergence bound}
    \begin{split}
        &\int_0^T
    \|\operatorname{div} v_t\|_\infty\,\de t \leq
    \varepsilon
    \sum_{j=-1} ^{\infty}
        \int_0^T\|\proj_j \left(\cA [f_t] \right)\|_{\infty} \,\de t
    \overset{\eqref{e:Bernstein inequality 1}}{\leq}
    C_d \cdot \varepsilon
    \sum_{j=-1} ^{\infty}
        2^{jd/2}
        \int_0^T\|\proj_j \left(\cA [f_t] \right)\|_2\,\de t
    \\
    &\leq
    C_d \cdot \varepsilon 
    \sum_{j=-1} ^{\infty}
        2^{(d/2-\sigma)j} \cdot \cY_T
    \leq
    C_{\sigma,d} \cdot \varepsilon \cdot \cY_T,
    \end{split}
\end{align}
where the last geometric series is bounded by a constant $C_{\sigma,d}$ because $\sigma >d/2$.

We apply $\proj_k$ to \eqref{e:wass grad flow perturbation 2}. Notice that $\proj_k$ commutes with $\cA$ and $\operatorname{div}v_t =-\varepsilon\cA[f_t]$. Thus, we obtain
\begin{align}\label{e:projected Eulerian equation}
    \partial_t (\proj_k f_t)
    - v_t\left(\proj_k f_t \right)
    +
    \cA \left[\proj_k f_t \right]
    =
    \cR_k(t),
\end{align}
where $\cR_k(t)
    \coloneqq
    \mathcal T_k[v_t,f_t]
    -\varepsilon\proj_k\left(f_t \cA[f_t] \right)$.
Then,
\begin{align*}
    \sup_{k\geq-1}2^{\sigma k}
    \int_0^T
        \|\cT_k[v_t,f_t]\|_2\,\de t \overset{\Cref{lem:torus dyadic transport}}{\leq} C_{\sigma,d} \cdot \cX_T \cdot \|v_t\|_{B_T^{\sigma+1}} \overset{\eqref{e:velocity Chemin bound}}{\leq}
    C_{\sigma,d} \cdot \varepsilon \cdot \cX_T \cdot \cY_T, 
\end{align*}
and
\begin{align*}
    \begin{split}
        &\sup_{k\geq-1}2^{\sigma k}
    \int_0^T
        \varepsilon \cdot \left\|\proj_k\left(f_t \cA[f_t] \right) \right\|_2\,\de t = 
    \varepsilon  \cdot \left\|f_t \cA[f_t] \right\|_{B_T^\sigma}
    \\ & \overset{\Cref{lem:Besov product lemma}}{\leq}
    C_{\sigma,d} \cdot \varepsilon \cdot  \left(\sup_{0 \leq t \leq T}\left\|f_t \right\|_{B^\sigma} \right) \cdot \left\|\cA[f_t] \right\|_{B_T^\sigma} = C_{\sigma,d} \cdot \varepsilon \cdot \cX_T \cdot \cY_T .
    \end{split}
\end{align*}
Consequently,
\begin{align}\label{e:Eulerian remainder bound}
    \sup_{k\geq-1}2^{\sigma k}
    \int_0^T\|\cR_k(t)\|_2\,\de t
    \leq
    C_{\sigma,d} \cdot \varepsilon \cdot \cX_T \cdot \cY_T .
\end{align}

Then, we estimate $\|\proj_k f_t \|_2$. Denote
\begin{align*}
    x_k(t)\coloneqq\|\proj_k f_t\|_2,
    \quad
    y(t)
    \coloneqq
    \frac{1}{2} \|\operatorname{div} v_t\|_\infty,
    \quad z_k(t)\coloneqq\|\cR_k(t)\|_2.
\end{align*}
By \eqref{e:projected Eulerian equation}, we have that
\begin{align}\label{e:exact block energy identity}
    \begin{split}
        \frac{1}{2}\frac{\de}{\de t}x_k^2
    +
    \langle\cA[\proj_k f_t],\proj_k f_t\rangle
    &=
    \frac{1}{2}
    \int_{M}
        \langle v_t, \nabla (\proj_k f_t)^2 \rangle \,\de x
    + \int_M \cR_k \cdot  (\proj_k f_t) \, \de x 
    \\ &= -\frac{1}{2}
    \int_{M}
        \operatorname{div}\left( v_t \right) \left|\proj_k f_t \right|^2  \,\de x
    +
    \int_M \cR_k \cdot  (\proj_k f_t) \, \de x.
    \end{split}
\end{align}
By \eqref{e:A block coercivity} and H\"older's inequality, we obtain
\begin{align}
    x_k'+C_s^{-1} \cdot \omega_k\cdot x_k
    &\leq
    y \cdot x_k+z_k,
    \label{e:upper scalar block inequality}
    \\
    x_k'+C_s \cdot \omega_k\cdot x_k
    &\geq
    -y \cdot x_k-z_k.
    \label{e:lower scalar block inequality}
\end{align}
These inequalities remain valid at times when $x_k=0$: one can apply
\eqref{e:exact block energy identity} to
$(x_k^2+\delta^2)^{1/2}$ and then let $\delta \to 0$. Now, let
\begin{align*}
    Y(t)\coloneqq\int_0^t y(r)\,\de r.
\end{align*}
For any $t \in [0,T]$, by \eqref{e:principal divergence bound}, $Y(t) \leq Y(T) \leq C_{\sigma,d} \cdot \varepsilon \cdot \cY_T$. Then, using \eqref{e:upper scalar block inequality}, a Gr\"onwall argument yields that 
\begin{align}\label{e:x_k(t) upper bound integral}
    \begin{split}
        x_k(t)
    &\leq
    e^{Y(T)}
    \left(
        e^{-C_s^{-1} \cdot \omega_k \cdot t}x_k(0)
        +
        \int_0^t
            e^{-C_s^{-1} \cdot \omega_k \cdot(t-r)}z_k(r)\,\de r
    \right)
    \\ &\overset{\eqref{e:principal divergence bound}}{\leq} e^{C_{\sigma,d} \cdot \varepsilon \cdot \cY_T}
    \left(
        e^{-C_s^{-1} \cdot \omega_k \cdot t}x_k(0)
        +
        \int_0^t
            e^{-C_s^{-1} \cdot \omega_k \cdot(t-r)}z_k(r)\,\de r
    \right).
    \end{split}
\end{align}
Thus, \eqref{e:x_k(t) upper bound integral} implies that
    \begin{align}\label{e:X_T Y_T integral bound 1}
        \sup_{0\leq t\leq T}x_k(t) \leq e^{C_{\sigma,d} \cdot \varepsilon \cdot \cY_T}
    \left(x_k(0) + \sup_{0\leq t\leq T} \int_0 ^t z_k(r)\,\de r \right) \leq e^{C_{\sigma,d} \cdot \varepsilon \cdot \cY_T}
    \left(x_k(0) +  \int_0 ^T z_k(r)\,\de r \right),
    \end{align}
and
    \begin{align}\label{e:X_T Y_T integral bound 2}
         \begin{split}
             \omega_k\int_0^Tx_k(t)\,\de t &\leq e^{C_{\sigma,d} \cdot \varepsilon \cdot \cY_T}
    \left(C_s \cdot x_k(0) +  \omega_k \int_0^T z_k(r) \int_r ^T
            e^{-C_s^{-1} \cdot \omega_k \cdot(t-r)}\,\de t \de r\right) 
            \\ &\leq C_{s} \cdot e^{C_{\sigma,d} \cdot \varepsilon \cdot \cY_T}
    \left(x_k(0) +  \int_0 ^T z_k(r)\,\de r  \right).
         \end{split}
    \end{align}
Multiplying \eqref{e:X_T Y_T integral bound 1} and \eqref{e:X_T Y_T integral bound 2} by $2^{\sigma k}$,
taking the supremum over $k$, and using
\eqref{e:A block coercivity} and
\eqref{e:Eulerian remainder bound}, we obtain that 
\begin{align}\label{e:Eulerian bootstrap inequality}
    \cX_T+\cY_T
    \leq
    C_{\ast} \cdot e^{ C_{\ast} \cdot \varepsilon \cdot \cY_T}
    \left(
        \cX_0+ C_{\ast} \cdot \varepsilon \cdot \cX_T \cdot \cY_T
    \right).
\end{align}
where $C_{\ast}>1$ is a constant depending only on $s,\sigma,d$.

Now, we can show that the solution $f_t$ exists for all $t \in [0,\infty)$ and \eqref{e:eulerian global bound} holds. Indeed, for the constant $C_{\ast}$ in \eqref{e:Eulerian bootstrap inequality}, we set 
\begin{align*}
    \calZ_T\coloneqq \cX_T+\cY_T,
    \quad
    M_0\coloneqq 4C_{\ast}\cX_0.
\end{align*}
If $\calZ_T\leq M_0$, then
\begin{align*}
    \cX_T \cY_T\leq \calZ_T^2\leq M_0^2.
\end{align*}
Choose $c_0$ in \eqref{e:eulerian smallness condition} sufficiently
small so that whenever $\varepsilon \cX_0\leq c_0$, we have that
\begin{align*}
    e^{C_{\ast}\varepsilon M_0}\leq \frac{3}{2} ,
    \quad
    C_{\ast} \varepsilon M_0^2\leq\frac{1}{3} \cX_0
\end{align*}
 Then
\eqref{e:Eulerian bootstrap inequality} improves $\calZ_T\leq M_0$ to
\begin{align*}
    \calZ_T
    \leq
    C_{\ast} \cdot \frac{3}{2}
    \left(
        \cX_0+\frac{1}{3}\cX_0
    \right)
    =
    2C_{\ast} \cX_0
    =
    \frac{1}{2}M_0.
\end{align*}
Because $\calZ_0 = \cX_0 < M_0$, and $\calZ_T \leq M_0$ also gives a uniform bound on $\sup_{0 \leq t \leq T}\|f_t\|_{\infty}$ by \eqref{e:Besov norm control infinity norm}, the standard first-exit argument (see for example \cite[Proposition 2.9]{chizat2026quantitative}) gives the existence of $f_t$ on $[0,\infty)$, and for any $T \geq 0$,
    \begin{align}\label{e:eulerian global bound 2}
         \cX_T+ \cY_T\leq  M_0 \leq 4 C_{\ast} \cX_0,
    \end{align}
which implies \eqref{e:eulerian global bound}.

We finally establish \eqref{e:dyadic persistence quantitative} and \eqref{e:dyadic persistence simple}. By \eqref{e:lower scalar block inequality},
\begin{align}\label{e:x_k(t) lower bound integral}
    \begin{split}
        x_k(t)
    &\geq 
    e^{-C_s\cdot \omega_kt-Y(t)}x_k(0)
    -
    \int_0^t
        e^{-C_s \cdot \omega_k(t-r)-(Y(t)-Y(r))}
        z_k(r)\,\de r
    \\ &\geq e^{-C_s\cdot \omega_kt-Y(t)}x_k(0)
    -
    \int_0^t
        z_k(r)\,\de r .
    \end{split}
\end{align}
Notice that
    \begin{align*}
        Y(t) \overset{\eqref{e:principal divergence bound}}{\leq} C_{\sigma,d} \cdot \varepsilon \cdot \cY_T \overset{ \eqref{e:eulerian global bound 2}}{\leq}
    C_{\sigma,s,d} \cdot \varepsilon \cdot  \cX_0,
    \end{align*}
and
    \begin{align*}
        \int_0^t
        z_k(r)\,\de r \overset{\eqref{e:Eulerian remainder bound}}{\leq}  C_{\sigma,d} \cdot \varepsilon \cdot \cX_T \cdot \cY_T \cdot 2^{-\sigma k} \overset{ \eqref{e:eulerian global bound 2}}{\leq} C_{\sigma,s,d} \cdot \varepsilon \cdot \cX_0 ^2 \cdot 2^{-\sigma k}
    \end{align*}
Then, by \eqref{e:x_k(t) lower bound integral}, we obtain
\begin{align}\label{e:dyadic persistence all time}
    x_k(t)
    \geq
    e^{-C_{\sigma,s,d}(\omega_kt+\varepsilon \cX_0)}
    x_k(0)
    -
     C_{\sigma,s,d} \cdot \varepsilon \cdot \cX_0 ^2 \cdot 2^{-\sigma k},
\end{align}
which implies \eqref{e:dyadic persistence quantitative}. For \eqref{e:dyadic persistence simple}, choose $\tau_0>0$ small and further decrease $c_0$ so that for the constant $C_{\sigma,s,d}$ in \eqref{e:dyadic persistence all time}, we have that
\begin{align*}
    e^{-C_{\sigma,s,d}\left(\tau_0+c_0\right)}\geq\frac34.
\end{align*}
Assume \eqref{e:dyadic persistence initial}, i.e., assume that there is a constant $\gamma >0$ such that
\begin{align*}
    x_k(0) \geq\gamma2^{-\sigma k},
    \qquad
    \varepsilon \cX_0^2\leq c_1\gamma,
\end{align*}
where $c_1$ is chosen small so that for the constant $C_{\sigma,s,d}$ in \eqref{e:dyadic persistence all time}, $C_{\sigma,s,d} \cdot c_1\leq1/4$. Then for
$0\leq t\leq\tau_0\omega_k^{-1}$, by \eqref{e:dyadic persistence all time}, 
\begin{align*}
    x_k(t) \geq
    \left(
        \frac{3}{4}\gamma-C_{\sigma,s,d} \cdot c_1 \gamma
    \right)2^{-\sigma k} \geq
    \frac{1}{2}\gamma 2^{-\sigma k},
\end{align*}
which proves
\eqref{e:dyadic persistence simple}.

\end{proof}


\subsection{Proof of \Cref{thm:wass -1-delta rate intro} and \Cref{cor:uniform perturbation Besov norms}}\label{s:proof of wass -1-delta rate and uniform-in-time perturbation}

We first prove \Cref{cor:uniform perturbation Besov norms} by combining \Cref{prop:eulerian uniform estimate} and the following semigroup estimate for the inhomogeneous linearized equation.

\begin{lemma}[Semigroup estimate]\label{lem:semigroup estimate}
    Take any $T\geq 0$, $s \geq 1$, and $\rho \in \bR$. Suppose $h \in B^{\rho}(M)$ and $H,G$ are two time-dependent functions on $[0,T] \times M$, which satisfy the equation
        \begin{align}\label{e:semigroup eq}
            \partial_t H_t + \cA[H_t] = G_t, \quad H_0 = h.
        \end{align}
    There is a constant $C_{s,M}$ only depending on $s,M$, such that
        \begin{align}\label{e:semigroup estimate}
            \|H\|_{\BB_T ^{\rho}}
\leq C_{s,M} \cdot \left(
\|h\|_{B^{\rho}} + \|G\|_{B_T ^\rho} \right).
        \end{align}
\end{lemma}
\begin{proof}
    By Duhamel's principle,
    \begin{align*}
        H(t)=e^{-t\cA}h+\int_0^t e^{-(t-\tau)\mathcal A}G(\tau)\,\de \tau,
    \end{align*}
    where $e^{-t\cA}h$ denotes the solution to the linearized equation \eqref{e:linearized flow} starting from $h$.
    
    We first estimate the $B^{\rho}$-norm of $H$. For any $k \geq -1$, according to the explicit form \eqref{e:linearized basis expansion} applied to $e^{-t\cA}h$ and $e^{-(t-\tau)\mathcal A}G(\tau)$, we have that 
    \begin{align*}
        \|\proj_k e^{-t\cA} h \|_2 \leq \|\proj_k h\|_2, \quad \left\|
\proj_k
\int_0^t e^{-(t-\tau)\cA}G(\tau) \,\de \tau
\right\|_2
\leq
\int_0^t
\|\proj_k G(\tau)\|_2\,\de \tau.
    \end{align*}
Therefore, 
    \begin{align*}
         \begin{split}
             &\sup_{t \in [0,T]} \|H(t,\cdot)\|_{B^{\rho}} = 
\sup_{t \in [0,T]} \sup_{k \geq -1} 2^{k\rho} \|\proj_k H(t)\|_2
\\ &\leq \|h\|_{B^{\rho}}
+
\sup_{k \geq -1} 2^{k\rho} \int_0^T \|\proj_k G(\tau)\|_2\,\de \tau =  \|h\|_{B^{\rho}}
+ \|G\|_{B_T ^\rho}.
         \end{split}
    \end{align*}

We next estimate $\left\|\cA [H]\right\|_{B_T ^\rho}$. According to \eqref{e:linearized basis expansion}, we can write
    \begin{align*}
        \cA e^{-t\cA}h = \sum_{ n = 1} ^{\infty} \lambda_n ^{1-s} e^{- \lambda_n ^{1-s} t} b_n \phi_n, \quad \text{ with } b_n \coloneqq \int_M h \phi_n \,\de x.
    \end{align*}
Because $\supp( \eta) \subseteq [1,4]$ and $s \geq 1$, for $k\geq 0$ and those eigenvalues such that $2^{2k} \leq \lambda_n \leq 16 \cdot 2^{2k}$, we have that $\lambda_n ^{1-s} e^{- \lambda_n ^{1-s} t} \leq 2^{2k(1-s)} e ^{- c_s 2^{2k(1-s)} t}$, where $c_s = 16 ^{1-s}$. Thus,
    \begin{align}\label{e:block projection}
        \|\proj_k \cA e^{-t\cA}h \|_2 \leq 2^{2k(1-s)} e ^{- c_s 2^{2k(1-s)} t} \| \proj_k h\|_2.
    \end{align}
Similarly, for $k=-1$, we have an inequality similar to \eqref{e:block projection} up to a constant $C = C_M$ because the first eigenvalue $\lambda_1 > 0$. Thus,
    \begin{align*}
        \begin{split}
            & \| \cA[e^{-t\cA}h]\|_{B_T ^\rho}
=
\sup_{k\geq -1} 2^{k\rho} \int_0^T
\|\proj_k\cA e^{-t\cA}h \|_2 \, \de t 
\leq C_{s,M} 
\sup_{k\ge -1}
2^{k\rho} \| \proj_k h\|_2
\int_0^T
2^{2k(1-s)} e ^{- c_s 2^{2k(1-s)} t} \, \de t 
\\  & \leq  C_{s,M} 
\sup_{k\ge -1}
2^{k\rho} \| \proj_k h\|_2
\int_0 ^{\infty} e^{- c_s t} \, \de t =
C_{s,M} \|h\|_{B^\rho}.
        \end{split}
    \end{align*}
Similarly, for every $k \geq -1$, using the linearity of the operators $\cA$ and $\proj_k$,
        \begin{align*}
        \begin{split}
            & \left\| \cA\left[\int_0^t e^{-(t-\tau)\mathcal A}G(\tau)\,\de \tau \right] \right\|_{B_T ^\rho} \leq \sup_{k \geq -1} 2^{k \rho} \int_0^T
\int_0 ^{t}
\|\proj_k \cA e^{-(t-\tau)\mathcal A}G(\tau) \|_2  \, \de \tau \de t 
\\  &= \sup_{k \geq -1} 2^{k \rho} \int_0^T
\int_\tau ^{T}
\|\proj_k \cA e^{-(t-\tau)\mathcal A}G(\tau) \|_2  \, \de t \de \tau 
\\  &\overset{\eqref{e:block projection}}{\leq} C_M \sup_{k \geq -1} 2^{k \rho} \int_0^T
\int_\tau ^{T}
2^{2k(1-s)} e ^{- c_s 2^{2k(1-s)} (t-\tau)} \| \proj_k G(\tau)\|_2  \, \de t \de \tau
\\  &\leq C_M \sup_{k \geq -1} 2^{k \rho} \int_0^T \| \proj_k G(\tau)\|_2
\int_0 ^{\infty}
2^{2k(1-s)} e ^{- c_s 2^{2k(1-s)} t}   \, \de t \de \tau
\\  &=C_{s,M} 
\sup_{k \geq -1} 2^{k \rho} \int_0^T \| \proj_k G(\tau)\|_2
 \de \tau  = C_{s,M} \|G\|_{B_T ^\rho}.
        \end{split}
    \end{align*}
Thus, we obtain that
    \begin{align*}
        \left\|\cA [H]\right\|_{B_T ^\rho} \leq C_{s,M} \left( \|h\|_{B^{\rho}}
+ \|G\|_{B_T ^\rho} \right).
    \end{align*}
This finishes the proof of \Cref{lem:semigroup estimate}.

\end{proof}

\begin{proof}[Proof of \Cref{cor:uniform perturbation Besov norms}]
For the term $\operatorname{div}(f_t v_t )$ in \eqref{e:wass grad flow perturbation 2}, we first have that for any $T \geq 0$,
    \begin{align}\label{e:source comparison bound}
        \|\operatorname{div}(f_t v_t )\|_{B_T^{\sigma-1}} \leq C_{\sigma,d} \cdot  \|f_t v_t \|_{B_T^\sigma} \overset{\eqref{e:Besov product lemma in time}}{\leq} C_{\sigma,d} \cdot \cX_T \cdot \|v_t\|_{B_T^\sigma}\overset{\eqref{e:velocity Chemin bound} \text{ and } \eqref{e:eulerian global bound}}{\leq} C_{\sigma,s,d} \cdot \varepsilon \cdot \cX_0 ^2 .
    \end{align}
Thus, by \eqref{e:wass grad flow perturbation} and \eqref{e:linearized flow in nonlinear section},
\begin{align}\label{e:ft-bar ft differential equation}
    \partial_t \left( f_t - \overline f_t \right)+\cA \left[ f_t - \overline f_t \right]= \operatorname{div}(f_t v_t ),
    \quad f_0 - \overline f_0 = 0.
\end{align}
Then, \Cref{lem:semigroup estimate} with $\rho=\sigma-1$ directly implies that for any $T \geq 0$,
\begin{align*}
    \|f_t-\overline f_t\|_{\BB_T^{\sigma-1}}
    \leq
    C_{s,d} \|\operatorname{div}(f_t v_t )\|_{B_T^{\sigma-1}}
    \overset{\eqref{e:source comparison bound}}{\leq}
    C_{\sigma,s,d} \cdot \varepsilon \cdot \cX_0 ^2.
\end{align*}
The constant is independent of $T$, which proves
\eqref{e:uniform perturbation Besov norms}.
\end{proof}

We finally turn the persistence estimate \Cref{prop:eulerian uniform estimate} into a slow-decay construction. The following interpolation inequality, which is a special case of \cite[Theorem 6.4.5 (1)]{bergh2012interpolation}, supplies the matching upper bound in \Cref{thm:wass -1-delta rate intro}. For completeness, we reproduce its proof here.

\begin{lemma}[Interpolation between $\dot H^{1-2s}$ and $B^\sigma$]
\label{lem:negative Sobolev Besov interpolation}
Let $M= \TT^d$. Let $s>1$, $\sigma>0$, and let $h\in L_0^2(M)\cap B^\sigma(M)$. Denote
\begin{align}\label{e:def interpolation theta}
    \theta\coloneqq
    \frac{\sigma+s}{\sigma+2s-1}\in(0,1).
\end{align}
Then, there is a constant $C$ only depending on $s,\sigma$, such that
\begin{align}\label{e:negative Sobolev Besov interpolation}
    \|h\|_{\dot H^{-s}}^2
    \leq
    C
    \|h\|_{B^\sigma}^{2(1-\theta)}
    \|h\|_{\dot H^{1-2s}}^{2\theta}.
\end{align}
\end{lemma}

\begin{proof}
Assume that $h \neq 0$. Recall the homogeneous Sobolev norms defined in \eqref{e:sobolev norm}. For every $N \in \bZ$, splitting the dyadic sum at the dyadic index $N$ gives
\begin{align}\label{e:interpolation Besov low-high parts}
    \begin{split}
        \|h\|_{\dot H^{-s}}^2
    &=
    \sum_{k\leq N}2^{-2sk}\|h_k\|_2^2
    +
    \sum_{k>N}2^{-2sk}\|h_k\|_2^2 
    \\ &\leq 2^{2(s-1)N} \sum_{k\leq N}2^{2(1-2s)k}\|h_k\|_2^2  + \|h\|_{B^\sigma} ^2 \sum_{k>N}2^{-2(s+\sigma)k}                             \\
    &\leq 2^{2(s-1)N} \|h\|_{\dot H^{1-2s}} ^2
    +
    2^{-2(s+\sigma)N} \|h\|_{B^\sigma} ^2 .
    \end{split}
\end{align}
If $\|h\|_{B^\sigma} \leq \|h\|_{\dot H^{1-2s}}$, then we pick $N=-2$. The $k \leq N$ part in \eqref{e:interpolation Besov low-high parts} is empty and \eqref{e:interpolation Besov low-high parts} gives $\|h\|_{\dot H^{-s}}^2  \leq C_{\sigma,s} \|h\|_{B^\sigma} ^2$, which directly implies \eqref{e:negative Sobolev Besov interpolation}. If instead $\|h\|_{B^\sigma} > \|h\|_{\dot H^{1-2s}}$, we choose a positive real number $N_1$ such that
    \begin{align*}
        2^{2(2s+\sigma-1)N_1} =  \|h\|_{B^\sigma} ^2  \cdot \|h\|_{\dot H^{1-2s}} ^{-2} .
    \end{align*}
We then pick $N = \lfloor N_1 \rfloor +1 \in \bZ_+$ and thus
    \begin{align}\label{e:choice N interpolation Besov}
        \|h\|_{B^\sigma} ^2  \cdot \|h\|_{\dot H^{1-2s}} ^{-2} \leq 2^{2(2s+\sigma-1)N} \leq C_{\sigma,s} \|h\|_{B^\sigma} ^2  \cdot \|h\|_{\dot H^{1-2s}} ^{-2}
    \end{align}
\eqref{e:negative Sobolev Besov interpolation} then follows from plugging \eqref{e:choice N interpolation Besov} into \eqref{e:interpolation Besov low-high parts}.

\end{proof}

\begin{proof}[Proof of \Cref{thm:wass -1-delta rate intro}]
The construction of the initial datum $f_0$ uses only Weyl's law and therefore
works on any closed connected manifold. By Weyl's law, there is a sufficiently large $k_0 = k_0(M) \in \bZ_+$, such that when $k \geq k_0$, there exists an $n_k$ such that the Laplacian eigenvalue $\lambda_{n_k}$ satisfies 
\begin{align*}
    2\cdot2^k
    \leq \sqrt{\lambda_{n_k}}
    \leq 3\cdot2^k.
\end{align*}
We denote the corresponding real-valued $L^2$-normalized eigenfunctions by $\phi_{n_k}$. By the construction of $\eta$ in the Littlewood--Paley decomposition, $\eta=1$ on $[2,3]$ and thus $\eta\left( 2^{-k} \sqrt{\lambda_{n_k}}\right) = 1$. The partition of unity
\eqref{e:partition of unity} gives
\begin{align}\label{e:chosen mode survives projection}
    \proj_k\phi_{n_k}=\phi_{n_k},
    \qquad
    \proj_j\phi_{n_k}=0 \quad\text{ for }j\ne k.
\end{align}
Choose
\begin{align}\label{e:choice f_0}
    f_0= \sum_{k\geq k_0}2^{-\sigma k}\phi_{n_k}.
\end{align}
Then for any $k \geq k_0$, $\proj_k f_0 = 2^{-\sigma k }\phi_{n_k}$, and clearly, $\|f_0\|_{B^\sigma}=1$. We now take $M=\TT^d$, as required by
\Cref{prop:eulerian uniform estimate}.

Choose $0<\varepsilon\leq \min\{c_0,c_1\}$ for the constants $c_0,c_1$ depending only on $s,\sigma,d$ in \Cref{prop:eulerian uniform estimate}, which ensures both assumptions 
\eqref{e:eulerian smallness condition} and \eqref{e:dyadic persistence initial} for $\gamma = 1$ in \Cref{prop:eulerian uniform estimate}. \Cref{prop:eulerian uniform estimate} also gives a global
solution $f_t$ to \eqref{e:wass grad flow perturbation} on $[0,\infty) \times M$ with initial datum \eqref{e:choice f_0}, which satisfies
\begin{align}\label{e:slow solution uniform bound}
    \sup_{t\geq0}\|f_t\|_{B^\sigma}
    +\sup_{T\geq0}\|\cA [f_t]\|_{B_T^\sigma}
    \leq C_{s,\sigma,d}.
\end{align}
After decreasing $\varepsilon$ once more, the embedding
$B^\sigma (M)\hookrightarrow L^\infty (M)$ in \eqref{e:Besov norm control infinity norm} and \eqref{e:slow solution uniform bound} ensure that for any $t \geq 0$, $u_t = 1+\varepsilon f_t \geq 1/2$, and thus $u_t$ is a positive probability density. After fixing an $\varepsilon$ depending only on $s,\sigma,d$, the energies satisfy $\mathsf{E}[u_t] - \mathsf{E}_{\min} = \varepsilon^2 \mathsf{E}[f_t]$ according to \eqref{e:inverse Laplacian energy formula} and \eqref{e:energy ut ft}. We then establish the two bounds in \eqref{e:wass -1-delta rate} for $\mathsf{E}[f_t]$.

We first prove the lower bound in \eqref{e:wass -1-delta rate}. Recall the $r \coloneqq 2(s-1)$ defined in \Cref{prop:eulerian uniform estimate}. Because $f_0$ satisfies \eqref{e:dyadic persistence initial} for $\gamma = 1$, \eqref{e:dyadic persistence simple} implies that
\begin{align}\label{e:persistent f_t blocks}
    \|\proj_k f_t\|_2
    \geq \frac{1}{2} \cdot 2^{-\sigma k}
    \qquad
    \text{whenever }k\geq k_0
    \text{ and }0\leq t\leq\tau_0 2^{rk}.
\end{align}
For every $t\geq0$, choose $k(t)\geq k_0$ to be the smallest positive integer for
which $(t+1)\leq\tau_0 2^{rk(t)}$.  Then
\begin{align}\label{e:k(t) choice persistent scale}
    2^{rk(t)} \leq \max\left\{ 2^{rk_0} , \,  2^r \tau_0 ^{-1} \cdot (t+1) \right\} \leq  C_{s,\sigma,d} (1+t).
\end{align}
Thus, for any $t \geq 0$,
\begin{align*}
    \|f_t\|_{\dot H^{-s}}^2
    \overset{\eqref{e:sobolev norm}}{\geq} 2^{-2sk(t)}\|f_{k(t)}(t)\|_2^2 \overset{\eqref{e:persistent f_t blocks}}{\geq} \frac{1}{4} \cdot 2^{-2(s+\sigma)k(t)} \overset{\eqref{e:k(t) choice persistent scale}}{\geq}  C ^{-1}  \cdot (1+t)^{-\frac{s+\sigma}{s-1}},
\end{align*}
where $C>1$ is a constant depending only on $s,\sigma,d$.

We then prove the upper bound in \eqref{e:wass -1-delta rate}. We first have the equivalence of $\mathsf{E}[f_t]$ and the homogeneous Sobolev norm $\|f_t\|_{\dot H^{-s}}^2$ in \eqref{e:equivalence sobolev norm energy}. Then, differentiation
along the Wasserstein gradient flow \eqref{e:wass grad flow} gives \eqref{e:wass gradient}, which implies that 
    \begin{align*}
        \frac{\de}{\de t}\mathsf{E}[f_t] = - \int_M |\cQ[f_t]|^2 \cdot u_t \,\de x \leq -\frac{1}{2} \int_M |\cQ[f_t]|^2  \,\de x \overset{\eqref{e:sobolev norm}}{\leq} -C_s ^{-1} \cdot \|f_t\|_{\dot H^{1-2s}}^2,
    \end{align*}
where the first inequality is because $u_t \geq 1/2$, and the second inequality is the equivalence of \eqref{e:sobolev norm} and the standard form of the Sobolev norm $\dot H^{1-2s}$ up to a constant $C_s > 1$, which uses fractional Laplacian and is exactly $\int_M |\cQ[f_t]|^2  \,\de x$. Using \Cref{lem:negative Sobolev Besov interpolation} and \eqref{e:slow solution uniform bound}, we obtain
    \begin{align}\label{e:diff ineql E ft}
        \frac{\de}{\de t}\mathsf{E}[f_t] \leq - C_{s,\sigma,d} \cdot \|f_t\|_{\dot H^{-s}} ^{2/\theta} \overset{\eqref{e:equivalence sobolev norm energy}}{\leq} - C_{s,\sigma,d} \cdot \mathsf{E}[f_t] ^{1/\theta} ,
    \end{align}
where $\theta = \frac{\sigma+s}{\sigma+2s-1} \in (0,1)$ is defined in \eqref{e:def interpolation theta}. Since $\mathsf{E}[f_0]$ is a constant depending only on $s,\sigma,d$, we obtain from \eqref{e:diff ineql E ft} that
    \begin{align*}
        (\mathsf{E}[f_t]) ^{1-1/\theta} \geq  C_{s,\sigma,d} \cdot \left(t+ (\mathsf{E}[f_0]) ^{1-1/\theta} \right) \geq C_{s,\sigma,d} \cdot (t+1),
    \end{align*}
which is exactly
    \begin{align*}
        \mathsf{E}[f_t] \leq C_{s,\sigma,d} \cdot (t+1)^{-\frac{\theta}{1-\theta}} \overset{\eqref{e:def interpolation theta}}{=} C_{s,\sigma,d} \cdot (1+t)^{-\frac{s+\sigma}{s-1}} .
    \end{align*}
This finishes the proof of \Cref{thm:wass -1-delta rate intro}.
    
\end{proof}


\appendix

\section{Auxiliary lemmas in harmonic analysis}\label{s:harmonic analysis lemmas}

The following Bernstein-type estimates can be found in \cite[Lemma 3.1]{burq2005multilinear}, \cite[Lemma 9.2]{feichtinger2016geometric}, and \cite{sogge1988concerning,smith2007p,zworski2012semiclassical}. For completeness, we state them in a form adapted to our notation. To obtain the forms in \Cref{lem:Bernstein inequalities}, we only need to notice that both $\bfS_\ell h$ and $\proj_\ell h$ are linear combinations of eigenfunctions with eigenvalues less than $2^{2(\ell + 2)}$.
\begin{lemma}[Bernstein Inequalities]\label{lem:Bernstein inequalities}
    Let $h \in L^2(M)$ and adopt the dyadic partition defined in \eqref{e:partition of L^2 function}.
        \begin{itemize}
            \item \cite[Lemma 3.1]{burq2005multilinear}: For any $\ell \geq -1$, we have that 
        \begin{align}\label{e:Bernstein inequality 1}
            \| \proj_\ell h\|_{\infty} \leq C \cdot 2^{(\ell+2) d/2} \|\proj_\ell h\|_2,
        \end{align}
    where the constant $C>0$ depends only on $d,M$.
        \item \cite[Lemma 9.2]{feichtinger2016geometric}: For any $\ell \geq -1$, any $\gamma \in \bN$, and any $q \in [1,\infty]$, we have that
        \begin{align}\label{e:Bernstein inequality 2}
            \| \nabla^{\gamma} \bfS_\ell h \|_{q} \leq C' \cdot 2^{(\ell+2) \cdot \gamma} \|  \bfS_\ell h \|_{q}, \quad \| \nabla^{\gamma} \proj_\ell h \|_{q} \leq C' \cdot 2^{(\ell+2) \cdot \gamma} \|  \proj_\ell h \|_{q},
        \end{align}
    where the constant $C'>0$ depends only on $\gamma,q,d,M$.
        \end{itemize}
\end{lemma}
A direct corollary of \eqref{e:Bernstein inequality 1} is the following:
\begin{corollary}\label{cor:Bernstein inequalities 2-1}
    Let $h \in L^2(M)$ and adopt the dyadic partition defined in \eqref{e:partition of L^2 function}. For any $\ell \geq -1$, we have that
        \begin{align}\label{e:Bernstein inequality 3}
            \|\proj_{\ell} h\|_2 \leq C \cdot 2^{\ell d/2} \|h\|_1,
        \end{align}
    where the constant $C>0$ depends only on $d,M$.
\end{corollary}
\begin{proof}
    The proof follows from a direct duality argument. Indeed,
        \begin{align*}
            \|\proj_{\ell} h\|_2 = \sup_{\|g\|_2 = 1 } \langle \proj_{\ell} h , \,  g \rangle = \sup_{\|g\|_2 = 1 } \langle  h , \,  \proj_{\ell} g \rangle \leq \sup_{\|g\|_2 = 1 }   \|h\|_1 \cdot \|\proj_{\ell} g\|_{\infty} \overset{\eqref{e:Bernstein inequality 1}}{\leq} C \cdot 2^{\ell d/2} \cdot \|h\|_1 .
        \end{align*}
\end{proof}

We also need the following almost-orthogonality of the smooth Littlewood--Paley projectors $\proj_k$.
\begin{lemma}[Almost-orthogonality of $\proj_k$]\label{lem:almost orthogonal littlewood paley}
    Let $h,g \in L^2\left( M \right)$ and adopt the dyadic partition defined in \eqref{e:partition of L^2 function} and \eqref{e:dyadic block and partial sum notation}. For any $k,j,\ell \geq -1$, and any $\xi >0 $, there is a constant $C>0$ depending only on $\xi , M$, such that
        \begin{align}\label{e:almost orthogonal 1}
            \|\proj_k \left( \bfS_{\ell -4} h \cdot \proj_\ell g \right)\|_2 \leq C \cdot 2^{-\xi(k-\ell)_+} \| \bfS_{\ell -4} h \|_{\infty} \|\proj_\ell g \|_2,
        \end{align}
    and
        \begin{align}\label{e:almost orthogonal 2}
            \|\proj_k \left( \proj_{j } h \cdot \proj_\ell g \right)\|_2 \leq C \cdot 2^{-\xi(k-\max\{j,\ell\})_+} \| \proj_{j} h \|_{\infty} \|\proj_\ell g \|_2,
        \end{align}
    where $(k-\ell)_+=\max\{k-\ell,0\}$.
\end{lemma}
\begin{proof} We only prove \eqref{e:almost orthogonal 1} and the proof of \eqref{e:almost orthogonal 2} is similar. When $k = -1$, clearly, 
    \begin{align*}
        \|\proj_{-1} \left( \bfS_{\ell -4} h \cdot \proj_\ell g \right)\|_2 \leq \|\bfS_{\ell -4} h \cdot \proj_\ell g \|_2 \leq \| \bfS_{\ell -4} h \|_{\infty} \|\proj_\ell g \|_2 .
    \end{align*}
When $k > -1$, we let $N = \lfloor \xi/2 \rfloor + 1 \in \bZ_+$. First, for any $h \in C^{\infty}(M)$, we adopt its spectral expansion \eqref{e:L^2 spectral expansion}. By \eqref{e:partition of unity projection}, we have that
    \begin{align*}
        \|\proj_k h \|_2 ^ 2 \leq \sum_{2^k \leq \sqrt{\lambda_n} \leq 4 \cdot 2^k} b_n ^2 \leq 2^{-4k N} \sum_{2^k \leq \sqrt{\lambda_n} \leq 4 \cdot 2^k} \left(\sqrt{\lambda_n } \right) ^{4N}  b_n ^2  \leq 2^{-4k N} \left \| \left(-\Delta \right)^{N} h\right\|_2 ^2.
    \end{align*}
Thus,
    \begin{align*}
        \|\proj_k \left( \bfS_{\ell -4} h \cdot \proj_\ell g \right)\|_2 \leq 2^{-2kN} \left \| \left(-\Delta \right)^{N} \left( \bfS_{\ell -4} h \cdot \proj_\ell g \right) \right\|_2.
    \end{align*}
Then, the Leibniz rule implies that for some $C_{\xi} >0$ depending on $\xi$ only,
    \begin{align*}
        \begin{split}
            &\left \| \left(-\Delta \right)^{N} \left( \bfS_{\ell -4} h \cdot \proj_\ell g \right) \right\|_2 \leq C_{\xi} \sum_{\gamma = 0} ^{2N} \left \| \nabla^{\gamma} \bfS_{\ell -4} h  \right\|_{\infty} \cdot \left \| \nabla^{2N- \gamma} \proj_\ell g  \right\|_{2}
            \\  & \overset{\eqref{e:Bernstein inequality 2}}{\leq} C_{\xi,M}  \sum_{\gamma = 0} ^{2N} 2^{\ell \gamma }\left \|  \bfS_{\ell -4} h  \right\|_{\infty} \cdot 2^{\ell (2N-\gamma)} \left \| \proj_\ell g  \right\|_{2} = C_{\xi,M} 2^{2N \ell  }\left \|  \bfS_{\ell -4} h \right\|_{\infty} \left \| \proj_\ell g  \right\|_{2} .
        \end{split}
    \end{align*}
Thus, combining the above two estimates, we obtain that when $k \geq \ell$,
    \begin{align*}
        \|\proj_k \left( \bfS_{\ell -4} h \cdot \proj_\ell g \right)\|_2 \leq C_{\xi,M} 2^{-2N(k-\ell)} \| \bfS_{\ell -4} h \|_{\infty} \|\proj_\ell g \|_2 \leq C_{\xi,M} 2^{-\xi(k-\ell)} \| \bfS_{\ell -4} h \|_{\infty} \|\proj_\ell g \|_2.
    \end{align*}
When $k < \ell$, we directly use the estimate
    \begin{align*}
        \|\proj_k \left( \bfS_{\ell -4} h \cdot \proj_\ell g \right)\|_2 \leq \| \bfS_{\ell -4} h \cdot \proj_\ell g \|_2 \leq  \| \bfS_{\ell -4} h \|_{\infty} \|\proj_\ell  g \|_2.
    \end{align*}
This proves \Cref{lem:almost orthogonal littlewood paley}.

\end{proof}





\bibliographystyle{abbrv} 
\bibliography{references}

\end{document}